\documentclass[12pt]{amsart}
\usepackage[utf8]{inputenc}
\usepackage[english]{babel}
\usepackage{amsmath, amssymb, amsthm, amscd,color,comment}
\usepackage{cancel}
\usepackage{cite}
\usepackage{alltt}
\usepackage[dvipsnames]{xcolor}
\usepackage{array}
\usepackage[small,bf,labelsep=period]{caption}
\usepackage{mathtools}
\usepackage{textcomp}
\usepackage{enumerate}
\newtheorem{theorem}{Theorem}[section]

\newtheorem{lemma}[theorem]{Lemma}

\newtheorem{corollary}[theorem]{Corollary}
\newtheorem{proposition}[theorem]{Proposition}

\begin{document}
	\setlength{\parindent}{0cm}	
	\title[Ruelle's inequality and Pesin's formula]{Ruelle's inequality and Pesin's formula for Anosov geodesic flows in non-compact manifolds}
	
	\thanks{{\bf Keywords}: Anosov geodesic flow, Jacobi field, Lyapunov exponents, Ruelle's inequality, Pesin's formula.}
	
	\thanks{{\bf Mathematics Subject Classification (2010)}: 37D40, 53C20. }
	
	\author{Alexander Cantoral}
	\address{Instituto de Matemática, Universidade Federal do Rio de Janeiro, CEP 21941-909, Rio de Janeiro, Brazil}
	\email{alexander.vidal@im.ufrj.br}
	
	\author{Sergio Romaña}
	\address{Instituto de Matemática, Universidade Federal do Rio de Janeiro, CEP 21941-909, Rio de Janeiro, Brazil}
	\email{sergiori@im.ufrj.br}
	\begin{abstract} In this paper we prove Ruelle's inequality for the geodesic flow in non-compact manifolds with Anosov geodesic flow and some assumptions on the curvature. In the same way, we obtain the Pesin's formula for Anosov geodesic flow in non-compact manifolds with finite volume.
	\end{abstract}
	
	\maketitle
	
	\section{Introduction}\noindent
Ruelle in \cite{ruelle} proved an important result in ergodic theory relating entropy and Lyapunov exponents. More precisely, if $f:M\rightarrow M$ is a $C^1$-diffeomorphism on a compact manifold and $\mu$ is an $f$-invariant probability measure on $M$, then 
\begin{align}\label{intro}
h_\mu(f)\le \int \sum_{\mathcal{X}_i(x)>0}\mathcal{X}_i(x)\cdot \dim(H_i(x))d\mu(x),
\end{align}
where $h_\mu(f)$ is the entropy, $\left\lbrace \mathcal{X}_i(x)\right\rbrace$ is the set of Lyapunov exponents at $x\in M$ and $\dim(H_i(x))$ is the multiplicity of $\mathcal{X}_i(x)$.
In situations involving non-compact manifolds, Ruelle's inequality may be compromised. For example, Riquelme in \cite{contra} constructed diffeomorphisms defined on non-compact manifolds with an invariant measure with positive entropy and the sum of the positive Lyapunov exponents was equal to zero. However, in recent years, certain findings have been achieved that, in particular situations, offer the possibility of verifying Ruelle's inequality in non-compact contexts. Liao and Qiu in \cite{liao} showed Ruelle's inequality for general Riemannian manifolds under an integrable condition. Riquelme in \cite{riquelme} showed Ruelle's inequality for the geodesic flow in manifolds with pinched negative sectional curvature with some condition about the derivatives of the sectional curvature.\\
The main goal of this work is to prove Ruelle's inequality for the geodesic flow on the unit tangent bundle of a non-compact manifold with Anosov geodesic flow and some assumptions on the curvature. More precisely,
\begin{theorem}
Let $M$ be a complete Riemannian manifold with Anosov geodesic flow. Assume that the curvature tensor and the derivative of the curvature tensor are both uniformly bounded. Then, for every $\phi^t$-invariant probability measure $\mu$ on $SM$, we have
\begin{align*}
	h_\mu(\phi)\le \int\limits_{SM} \sum_{\mathcal{X}_i(\theta)>0}\mathcal{X}_i(\theta)\cdot\dim (H_i(\theta))d\mu(\theta).
\end{align*}
\end{theorem}
We can see that this result generalizes what Riquelme demonstrated in \cite{riquelme} since the Anosov geodesic flows encompass the manifolds with pinched negative curvature.\\
The question arises as to under what conditions equality can be achieved in \eqref{intro}. For example, when the manifold is compact, the diffeomorphism is $C^{1+\alpha}$ and the measure is absolutely continuous with respect	to the Lebesgue measure, Pesin showed in \cite{igual} that \eqref{intro} is actually an equality, called Pesin's formula. Our second result deals with the equality case of Theorem 1.1. In this case, we suppose that the manifold has finite volume.
\begin{theorem}
Let $M$ be a complete Riemannian manifold with finite volume and Anosov geodesic flow, where the flow is $C^1$-Hölder. Assume that the curvature tensor and the derivative of the curvature tensor are both uniformly bounded. Then, for every $\phi^t$-invariant probability measure $\mu$ on $SM$ which is absolutely continuous relative to the Lebesgue measure, we have
\begin{align*}
	h_\mu(\phi)=  \int\limits_{SM} \sum_{\mathcal{X}_i(\theta)>0}\mathcal{X}_i(\theta)\cdot\dim (H_i(\theta))d\mu(\theta).
\end{align*}
\end{theorem}
\subsection*{Structure of the Paper:} 
In section 2, we introduce the notations and geometric tools that we use in the paper. In section 3, we prove the existence of Oseledec's decomposition for the flow at time $t=1$. In section 4, we explore certain results that will allow us to deal with the challenge of non-compactness of the manifold. Using the strategies exhibited in \cite{pesin} to prove the Ruelle's inequality for diffeomorphisms in the compact case, we prove Theorem 1.1 in section 5. Finally, in section 6 we prove Theorem 1.2 using techniques applied by Mañe in \cite{mane}.
	
\section{Preliminaries and notation}
\noindent
Throughout this paper, $M=(M,g)$ will denote a complete Riemannian manifold without boundary of dimension $n\ge 2$, $TM$ is the tangent bundle, $SM$ its unit tangent bundle and $\pi:TM\rightarrow M$ will denote the canonical projection, that is, $\pi(x,v)=x$ for $(x,v)\in TM$.
	\subsection{Geodesic flow}
	Given $\theta=(x,v)\in TM$, we denote by $\gamma_\theta$ the unique geodesic with initial conditions $\gamma_\theta(0)=x$ and $\gamma'_\theta(0)=v$. The geodesic flow is a family of $C^\infty$-diffeomorphisms $\phi^t:TM\rightarrow TM$, where $t\in \mathbb{R}$, given by
	\begin{align*}
		\phi^t(\theta)=(\gamma_\theta(t), \gamma'_\theta(t)).
	\end{align*}
	Since geodesics travel with constant speed, we have that $\phi^t$ leaves $SM$ invariant. The geodesic flow generates a vector field $G$ on $TM$ given by
	\begin{align*}
		G(\theta)=\left. \dfrac{d}{dt}\right|_{t=0} \phi^t(\theta)=\left. \dfrac{d}{dt}\right|_{t=0}\left( \gamma_\theta(t),\gamma'_\theta(t)\right) . 
	\end{align*}
	For each $\theta=(x,v)\in TM$, let $V$ be the vertical subbundle of $TM$ whose fiber at $\theta$ is given by $V_\theta=\ker d\pi_\theta$. Let $K: TTM\rightarrow TM$ be the connection map induced by the Riemannian metric (see \cite{paternain}) and denotes by $H$ the horizontal subbundle of $TM$ whose fiber at $\theta$ is given by $H_\theta=\ker K_\theta$. The maps $\left. d \pi_\theta\right|_{H_\theta}:H_\theta\rightarrow T_xM$ and $\left. K_\theta\right|_{V_\theta}: V_\theta\rightarrow T_xM$ are linear isomorphisms. This implies that $T_\theta TM=H_\theta\oplus V_\theta$ and the map $j_\theta:T_\theta TM\rightarrow T_xM\times T_xM$ given by
\begin{align}\label{jota}
		j_\theta(\xi)=(d\pi_\theta(\xi), K_\theta(\xi))
	\end{align}
	is a linear isomorphism. Furthermore, we can identify every element $\xi\in T_\theta TM$ with the pair $j_\theta(\xi)$. Using the decomposition $T_\theta TM=H_\theta \oplus V_\theta$, we endow the tangent bundle $TM$ with a special Riemannian metric that makes $H_\theta$ and $V_\theta$ orthogonal. This metric is called the Sasaki metric and it's given by
	\begin{align*}
		\left\langle \xi, \eta \right\rangle_\theta=\left\langle d\pi_\theta(\xi), d\pi_\theta(\eta)\right\rangle_x + \left\langle K_\theta(\xi),K_\theta(\eta)\right\rangle_x . 
	\end{align*}
From now on, we work with the Sasaki metric restricted to the unit tangent bundle $SM$. To begin with, it is valid to ask if $SM$ is a complete Riemannian manifold with this metric. 
\begin{lemma} Let $M$ be a complete Riemannian manifold. Then $SM$ is a complete metric space with the Sasaki metric.
\end{lemma}
\begin{proof}
Let $\theta, \omega\in SM$ and $\gamma:[0,1]\rightarrow SM$ be a curve joining $\theta$ and $\omega$. By the identification \eqref{jota} we can write
\begin{align*}
	l(\gamma)&=\int_0^1 \left\| \gamma'(t)\right\|dt \\
	&=\int_0^1 \left( \left\| d\pi_{\gamma(t)}( \gamma'(t))\right\|^2 + \left\| K_{\gamma(t)}( \gamma'(t))\right\|^2\right)^{1/2}  dt\\
	&\ge \int_0^1  \left\| d\pi_{\gamma(t)}( \gamma'(t))\right\| dt\\
	&=\int_0^1 \left\| \left( \pi\circ \gamma\right)'(t) \right\|dt\\
	&=l(\pi\circ \gamma).
\end{align*}
This implies that
\begin{align}\label{comple}
 d(\theta,\omega)\ge d(\pi(\theta),\pi(\omega))   
\end{align}
for any two points $\theta,\omega\in SM$. Let $\left\lbrace (p_n,v_n)\right\rbrace_{n\in \mathbb{N}}$ be a Cauchy sequence in $SM$. By \eqref{comple} we have that $\left\lbrace p_n\right\rbrace_{n\in \mathbb{N}}$ is a Cauchy sequence in $M$. Since $M$ is complete, there is $p\in M$ such that $\displaystyle\lim_{n\rightarrow +\infty}p_n=p$. If we consider the compact set $X=\left\lbrace (q,v)\in SM: d(q,p)\le 1\right\rbrace$, for $n\ge n_0$ we have that $(p_n,v_n)\in X$ and therefore the Cauchy sequence converges in $SM$.
\end{proof}
\noindent
The sectional curvature of $SM$ with the Sasaki metric can be calculated from the curvature tensor and the derivative of the curvature tensor of $M$ as explained in \cite{kowalski}: Let $\Pi$ be a plane in $T_{(x,v)}SM$ and choose an orthonormal basis $\left\lbrace (v_1,w_1), (v_2,w_2)\right\rbrace$ for $\Pi$ satisfying $\left\| v_i\right\|^2+\left\| w_i\right\|^2=1 $, for $i=1,2$, and $\left\langle v_1,v_2 \right\rangle =\left\langle w_1,w_2\right\rangle=0$. Then the Sasaki sectional curvature of $\Pi$ is given by
\begin{align}\label{curvaturaSM}
	K_{Sas}(\Pi)=&\left\langle R_x(v_1,v_2)v_1,v_2 \right\rangle  +3\left\langle R_x(v_1,v_2)w_1,w_2 \right\rangle 
	+\left\| w_1\right\|^2 \left\| w_2\right\|^2 \nonumber \\ &-\dfrac{3}{4}\left\| R_x(v_1,v_2)v\right\|^2+\dfrac{1}{4}\left\|R_x(v,w_2)v_1 \right\|^2+\dfrac{1}{4}\left\|R_x(v,w_1)v_2 \right\|^2 \nonumber \\
	&+\dfrac{1}{2}\left\langle R_x(v,w_1)w_2,R_x(v,w_2)v_1 \right\rangle - \left\langle R_x(v,w_1)v_1,R_x(v,w_2)v_2 \right\rangle \nonumber \\
	&+\left\langle (\nabla_{v_1}R)_x(v,w_2)v_2,v_1 \right\rangle + \left\langle (\nabla_{v_2}R)_x(v,w_1)v_1,v_2 \right\rangle   .  
\end{align}
This equality shows that if the curvature tensor of $M$ and its derivatives are bounded, then the sectional curvature of $SM$ with the Sasaki metric is also bounded. This property is crucial as it allows us to compare volumes between subsets of $TSM$ and subsets of $SM$ using the exponential map of $SM$ (see Lemma 5.3).\\

The types of geodesic flows discussed in this paper are the Anosov geodesic flows, whose definition follows below.\\
We say that the geodesic flow $\phi^t:SM\rightarrow SM$ is of Anosov type if $T(SM)$ has a continuous splitting $T(SM)=E^s\oplus \left\langle G\right\rangle \oplus E^u$ such that
 \begin{eqnarray*}
d\phi^t_{\theta} (E^{s(u)}(\theta)) &=& E^{s(u)}(\phi^t(\theta)),\\
\left\|d\phi^t_{\theta}\big{|}_{E^s} \right\| &\leq& C \lambda^{t},\\
\left\|d\phi^{-t}_{\theta}\big{|}_{E^u}\right\| &\leq& C \lambda^{t},
\end{eqnarray*}
	for all $t\ge 0$ with $C>0$ and $\lambda\in (0,1)$, where $G$ is the geodesic vector field. It's known that if the geodesic flow is Anosov, then the subspaces $E^s(\theta)$ and $E^u(\theta)$ are Lagrangian for every $\theta\in SM$ (see \cite{paternain} for more details).
\subsection{Jacobi fields} To study the differential of the geodesic flow with geometric arguments, let us recall the definition of a Jacobi field. A vector field $J$ along a geodesic $\gamma$ of $M$ is a Jacobi field if it satisfies the Jacobi equation
\begin{align}\label{jacobi}
J''(t)+R(\gamma'(t),J(t))\gamma'(t)=0,
\end{align}
where $R$ denotes the curvature tensor of $M$ and $"'"$ denotes the covariant derivative along $\gamma$. A Jacobi field is determined by the initial values $J(t_0)$ and $J'(t_0)$, for any given $t_0\in \mathbb{R}$. If we denote by $S$ the orthogonal complement of the subspace spanned by $G$, for every $\theta\in SM$, the map $\xi\rightarrow J_\xi$ defines an isomorphism between $S(\theta)$ and the space of perpendicular Jacobi fields along $\gamma_\theta$, where $J_\xi(0)=d\pi_\theta(\xi)$ and $J'_\xi(0)=K_\theta(\xi)$.
The differential of the geodesic flow is determined by the behavior of the Jacobi fields and, therefore, by the curvature. More precisely, for $\theta\in SM$ and $\xi\in T_\theta SM$ we have (in the horizontal and vertical coordinates)
\begin{align*}
		d\phi^t_\theta(\xi)=(J_\xi(t), J'_\xi(t)), \hspace{0.5cm} t\in \mathbb{R}.
	\end{align*}
In the context of an Anosov geodesic flow, if $\xi\in E^s(\theta)$ (respectively, $\xi\in E^u(\theta)$), the Jacobi field associated $J_\xi(t)$ is called a stable (respectively, unstable) Jacobi field along $\gamma_\theta(t)$.\\
The following proposition allows us to uniformly limit the derivative of the exponential map from certain conditions on the curvature of the manifold.
\begin{proposition}
Let $N$ be a complete Riemannian manifold and suppose that the curvature tensor is uniformly bounded. Then there exists $t_0>0$ such that for all $x\in N$ and for all $v,w\in T_xN$ with $\left\| v\right\| =\left\| w\right\| =1$ we have
	\begin{align*}
		\left\| d(exp_x)_{tv}w\right\| \le \dfrac{5}{2}, \hspace{0.5cm} \forall \left| t\right| \le t_0.
	\end{align*}
\end{proposition}
\begin{proof}
If $w\in \left\langle v\right\rangle$, then $w=v$ or $w=-v$. In both cases, by Gauss Lemma (see \cite{lee}) we have that
\begin{align*}
	\left\| d(exp_x)_{tv}w\right\|^2&=\left\langle d(exp_x)_{tv}w, d(exp_x)_{tv}w\right\rangle \\
	&=\left\langle d(exp_x)_{tv}v, d(exp_x)_{tv}v\right\rangle\\
	&=\dfrac{1}{t^2}\left\langle d(exp_x)_{tv}tv, d(exp_x)_{tv}tv\right\rangle\\
	&=\dfrac{1}{t^2}\left\langle tv,tv\right\rangle \\
	&=1.
\end{align*}
Now assume that $w\in \left\langle v\right\rangle^\perp$. Consider the Jacobi field
\begin{align*}
	J(t)=d(exp_x)_{tv}tw, \hspace{0.3cm} t\in [-1,1]
\end{align*}
with initial conditions $J(0)=0$ and $J'(0)=w$. By Lemma 8.3 of \cite{burns} there exists $t_0>0$, independent of the point $x$, such that
\begin{align*}
	\left\| d(exp_x)_{tv}w\right\| =\dfrac{\left\|J(t) \right\|}{\left| t\right| }\le \dfrac{3}{2}, \hspace{0.5cm} \forall t\in(-t_0,t_0)\setminus \left\lbrace 0 \right\rbrace.
\end{align*}
As $T_xN=\left\langle v\right\rangle + \left\langle v\right\rangle^\perp$, the last inequality completes the proof.
\end{proof}
\subsection{No conjugate points}
Let $\gamma$ be a geodesic joining $p,q\in M$, $p\neq q$. We say that $p,q$ are conjugate along $\gamma$ if there exists a non-zero Jacobi field along $\gamma$ vanishing at $p$ and $q$. A manifold $M$ has no conjugate points if any pair of points are not conjugate. This is equivalent to the fact that the exponential map is non-singular at every point of $M$. There are examples of manifolds without conjugate points obtained from the hyperbolic behavior of the geodesic flow. In \cite{klin}, Klingenberg proved that a compact Riemannian manifold with Anosov geodesic flow has no conjugate points. Years later, Mañé (see \cite{mane2}) generalized this result to the case of manifolds of finite volume. In the case of infinite volume, Melo and Romaña in \cite{nocon} extended the result of Mañé over the assumption of sectional curvature bounded below and above. These results show the relationship that exists between the geometry and dynamic of an Anosov geodesic flow.\\
	
Let $M$ be a complete Riemannian manifold without conjugate points and sectional curvature bounded below by $-c^2$, for some $c>0$. When the geodesic flow $\phi^t:SM\rightarrow SM$ is of Anosov type, Bolton in \cite{bolton} showed that there is a positive constant $\delta$ such that, for every $\theta\in SM$, the angle between $E^s(\theta)$ and $E^u(\theta)$ is greater than $\delta$. Moreover, Eberlein in \cite{eber} showed that 
\begin{itemize}
	\item [1.] $\left\| K_\theta(\xi)\right\| \le c\left\| d\pi_\theta(\xi)\right\| $ for every $\xi\in E^s(\theta)$ or $E^u(\theta)$, where $K:TTM \rightarrow TM$ is the connection map.
	\item [2.] If $\xi\in E^s(\theta)$ or $E^u(\theta)$, then 	$J_\xi(t)\neq 0$ for every $t\in \mathbb{R}$.
\end{itemize}
\subsection{Lyapunov exponents}
Let $(M.g)$ be a Riemannian manifold and $f:M\rightarrow M$ a $C^1$-diffeomorphism. The point $x$ is said to be (Lyapunov-Perron) regular if there exist numbers $\left\lbrace\mathcal{X}_i(x) \right\rbrace_{i=1}^{l(x)}$, called Lyapunov exponents, and a decomposition of the tangent space at $x$ into $T_xM=\bigoplus_{i=1}^{l(x)}H_i$ such that for every vector $v\in H_i(x)\setminus \left\lbrace 0\right\rbrace $, we have  
\begin{align*}
	\lim_{n\rightarrow \pm \infty} \dfrac{1}{n}\log \left\|df^n_xv \right\|=\mathcal{X}_i(x) 
\end{align*}
and
\begin{align*}
	\lim_{n\rightarrow \pm \infty} \dfrac{1}{n}\log\left| \det\left( df^n_x\right) \right|=\sum_{i=1}^{l(x)}  \mathcal{X}_i(x)\cdot\dim (H_i(x)). 
\end{align*}
Let $\Lambda$ be the set of regular points. By Oseledec's Theorem (see \cite{oseledec}), if $\mu$ is an $f$-invariant probability measure on $M$ such that $\log^+ \left\| df^{\pm 1} \right\|$ is $\mu$-integrable, then the set $\Lambda$ has full $\mu$-measure. Moreover, the functions $x\rightarrow \mathcal{X}_i(x)$ and $x\rightarrow \dim(H_i(x))$ are $\mu$-measurable and $f$-invariant. In particular, if $\mu$ is ergodic, they are $\mu$-almost everywhere constant.
\section{Existence of Lyapunov exponents}
In this section, we will prove that when the geodesic flow is Anosov and the sectional curvature is bounded below, the norm $\left\| d\phi^{\pm 1}_\theta\right\| $ is bounded by a positive constant independent of $\theta$. This boundedness is crucial as it ensures, for a given probability measure, the existence of Lyapunov exponents by Oseledec's Theorem. More precisely,
\begin{theorem}
Let $M$ be a complete Riemannian manifold without conjugate points, sectional curvature bounded below by $-c^2$, for some $c>0$, and $\mu$ an $\phi^t$-invariant probability measure in $SM$. If the geodesic flow is of Anosov type, then $\log\left\| d\phi^{\pm 1}\right\|\in L^1(\mu)$.
\end{theorem}
Before giving a proof of Theorem 3.1, it is essential to establish the following two lemmas.
\begin{lemma}
Let $M$ be a complete Riemannian manifold without conjugate points, sectional curvature bounded below by $-c^2$, for some $c>0$, and geodesic flow of Anosov type. For every $\theta\in SM$, there exists a constant $P>0$ such that for every $\xi\in E^s(\theta)$, $\eta\in E^u(\theta)$ with $\left\| \xi\right\|=\left\| \eta\right\| =1$, we have
\begin{align*}
\left\|J_\eta(1) \right\|\le P \hspace{0.2cm}\text{ and } \hspace{0.2cm} \left\|J_\xi(-1) \right\|\le P.
\end{align*}
\end{lemma}	
\begin{proof}
Fix $\theta\in SM$ and let $\eta\in E^u(\theta)$ with $\left\| \eta\right\| =1$. Consider a stable Jacobi field $J_s$ along $\gamma_{\theta}$ such that $J_\eta(0)=J_s(0)$ and put $\omega=(J_s(0),J'_s(0))$. By item 1 of section 2.3 we have
\begin{align*}
\left\| J'_s(0)\right\|\le c \left\|J_s(0) \right\|=  c \left\|J_\eta(0) \right\|\le c 
\end{align*}
and
\begin{align*}
	\left\| \omega\right\|^2=\left\| J_s(0)\right\|^2  +\left\|J_s'(0) \right\|^2\le 1+c^2.
\end{align*}
Define the Jacobi field $J(t)= J_\eta(t)-J_s(t)$. We can see that $J$ is a perpendicular Jacobi field along $\gamma_{\theta}$ satisfying $J(0)=0$. By Rauch's comparison Theorem (see \cite{lee}) we have that
\begin{align}\label{unstable}
\left\| J(1)\right\|\le \dfrac{\sinh c}{c} \left\|J'(0)\right\| . 
\end{align} 	
Since the geodesic flow is Anosov and $\omega\in E^s(\theta)$,
\begin{align}\label{anosov}
\left\| J_s(1)\right\|\le \left\|d\phi^1_\theta (\omega) \right\|\le C\lambda\left\|\omega \right\|\le C\lambda\sqrt{1+c^2}.   
\end{align}
From \eqref{unstable} and \eqref{anosov} we have that
\begin{align*}
	\left\| J_\eta(1)\right\|&\le \left\| J(1)\right\| + \left\|J_s(1) \right\|\\
	&\le \dfrac{\sinh c}{c} \left\| J'(0)\right\| + C\lambda\sqrt{1+c^2}\\
	&\le \dfrac{\sinh c}{c}\left(\left\| J'_\eta(0)\right\| + \left\|J'_s(0) \right\|   \right)  +C\lambda\sqrt{1+c^2}\\
	&\le \left( \dfrac{1+c}{c}\right) \sinh c +C\lambda\sqrt{1+c^2}:=P_1. 
\end{align*}
Using the same technique for the stable case, there exists $P_2>0$ such that 
\begin{align*}
\left\| J_\xi(-1)\right\| \le P_2
\end{align*}
for every $\xi\in E^s(\theta)$ with $\left\| \xi\right\| =1$. Considering $P=\max\left\lbrace P_1,P_2\right\rbrace $, the conclusion of the lemma follows.
\end{proof}
We know that, with the hypothesis of Theorem 3.1, there exists a constant $\delta>0$ such that the angle between the stable and unstable subspaces is uniformly bounded below by $\delta$. As a direct consequence of this result, we have the following lemma.
\begin{lemma}
Let $M$ be a complete Riemannian manifold without conjugate points, sectional curvature bounded below by $-c^2$, for some $c>0$, and geodesic flow of Anosov type. Define the function $f:SM\rightarrow \mathbb{R}$ as
\begin{align*}
	f(\theta)=\sup \left\lbrace \left|\left\langle \xi,\eta \right\rangle  \right|: \xi\in E^s(\theta), \eta\in E^u(\theta), \left\| \xi\right\|=\left\| \eta\right\| =1  \right\rbrace .
\end{align*}
Then there exists $Q>0$ such that
\begin{align*}
\sup_{\theta\in SM} f(\theta)\le Q<1.
\end{align*}
\end{lemma}
\textbf{Proof of Theorem 3.1.} Fix $\theta \in SM$ and consider $\xi\in T_\theta SM$ with $\left\| \xi\right\| =1$. Since the geodesic flow is of Anosov type, we can write
\begin{align*}
\xi=s\xi_1+r\xi_2+\xi_3,
\end{align*}
where $\xi_1\in E^s(\theta)$, $\xi_2\in E^u(\theta)$ and $\xi_3\in \left\langle G(\theta)\right\rangle $ with $\left\|\xi_1 \right\| =\left\|\xi_2 \right\|=1$. Then
\begin{align*}
1=\left\|s\xi_1+r\xi_2 \right\|^2 + \left\| \xi_3\right\|^2.  
\end{align*} 
This implies that $\left\| \xi_3\right\| \le 1$ and $\left\| s\xi_1+r\xi_2\right\| \le 1$. We have
\begin{align*}
\left\| s\xi_1+r\xi_2\right\|^2=s^2+r^2+2sr\left\langle \xi_1,\xi_2 \right\rangle \le 1.
\end{align*}
It follows from Lemma 3.3 that the regions 
\begin{align*}
E_\beta=\left\lbrace (s,r):s^2+r^2+2sr\beta\le 1\right\rbrace 
\end{align*}
with $-Q\le \beta \le Q$ are bounded ellipses. If we consider $L=\frac{\text{diam}(E_Q)}{2}+1>0$, the ball $B$ centered in $0$ and radius $L$ contains these ellipses (see Figure 1). In particular, the parameters $s,r$ are bounded, that is $\left| s\right|, \left| r\right|\le L$. By Lemma 3.2 we have that
\begin{align*}
	\left\|d\phi^1_\theta(\xi_2) \right\|&=\sqrt{\left\| J_{\xi_2}(1)\right\|^2 + \left\|J'_{\xi_2}(1) \right\|^2  }\\
	&\le \sqrt{1+c^2}\left\| J_{\xi_2}(1)\right\|\\
	&\le \sqrt{1+c^2}P.
\end{align*}
\begin{figure}[h]
	\centering
\includegraphics[width=0.5\textwidth]{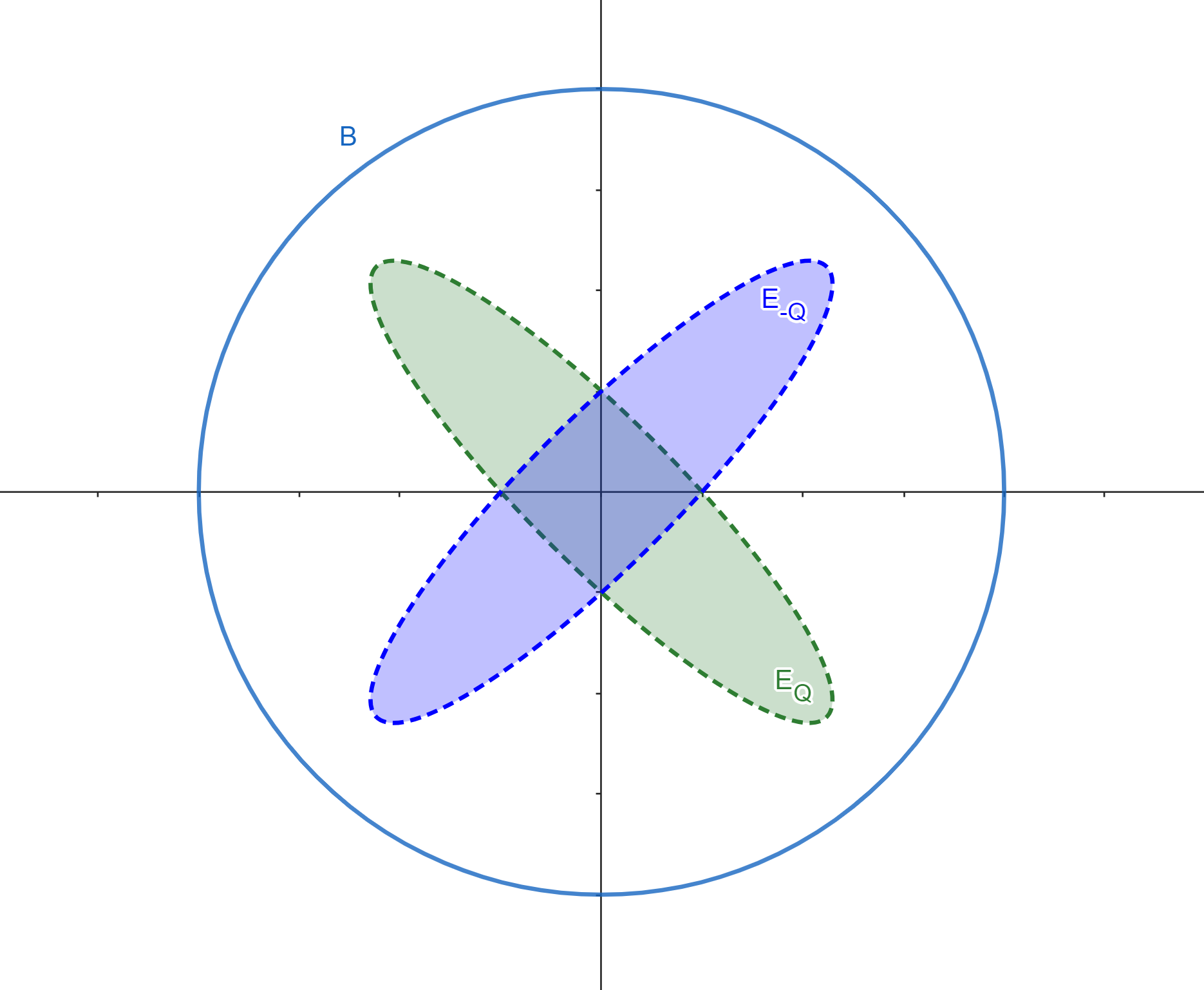}
\caption{Bounded ellipses for $Q<1$.}
\end{figure}\\
Then
\begin{align}\label{cotasup}
\left\|d\phi^1_\theta(\xi) \right\| &\le \left| s\right|\left\|d\phi^1_\theta(\xi_1) \right\| + \left| r\right|\left\|d\phi^1_\theta(\xi_2) \right\| + \left\|d\phi^1_\theta(\xi_3) \right\| \nonumber \\ 
&\le \left| s\right| C\lambda + \left| r\right|\sqrt{1+c^2}P+1 \nonumber \\ 
&\le LC\lambda + L\sqrt{1+c^2}P+1
\end{align}
for every $\xi\in T_\theta SM$ with $\left\| \xi\right\|  =1$. This implies that $\left\|d\phi^1_\theta \right\|$ is bounded and therefore the function $\log\left\|d\phi^1 \right\| $ is $\mu$-integrable, since the constants $L$ and $P$ are independent of the point $\theta$. Using the second inequality of Lemma 3.2 we obtain that $\log\left\| d\phi^{-1}\right\| $ is $\mu$-integrable.$\hfill\square$
\section{Consequences of a geodesic flow being of Anosov type}
In this section, we explore some results, based on the hyperbolicity of a geodesic flow, that will allow us to address the challenge of the non-compactness of the manifold in the proof of Ruelle's inequality. \\

From now on, let us assume that $M$ is a complete Riemannian manifold without conjugate points, sectional curvature bounded below by $-c^2$, for some $c>0$, and the geodesic flow $\phi^t:SM\rightarrow SM$ is of Anosov type. For every $\omega\in T_\theta SM$ we can write
\begin{align*}
\omega=\omega^s+\omega^u+\omega^c,
\end{align*}
where $\omega^s\in E^s(\theta)$, $\omega^u\in E^u(\theta)$ and $\omega^c\in \left\langle G(\theta)\right\rangle$.
\begin{lemma}
For $m\in \mathbb{N}$ large enough, there is $\tau_1>1$ such that for every $\theta\in SM$
\begin{align*}
\left\| d\phi^m_\theta\right\| \le \tau_1\left\| d\phi^m_\theta(\eta)\right\| 
\end{align*}
for some $\eta\in E^u(\theta)$ with $\left\| \eta\right\| =1$.
\end{lemma}
\begin{proof}
Fix $\theta\in SM$ and let $\omega=\omega^s+\omega^u+\omega^c\in T_\theta SM$ with $\left\| \omega\right\| =1$. This implies that $\left\|\omega^s+\omega^u \right\|\le 1 $ and $\left\|\omega^c \right\|\le 1 $. Moreover, we know that $\left\| \omega^s\right\| \le L$ and $\left\| \omega^u\right\| \le L$ (see Section 3). Consider $m\in \mathbb{N}$ large enough such that $C\lambda^m<1/2$.\\
Case 1: $\omega^u=0$.\\
Since the geodesic flow is Anosov we have that
\begin{align*}
\left\|d\phi^m_\theta(\omega) \right\|&\le \left\| d\phi^m_\theta(\omega^s)\right\| +\left\|d\phi^m_\theta(\omega^c) \right\|  \\
&\le  C\lambda^m + 1\\
&<C^{-1}\lambda^{-m}\\
&\le \left\| d\phi^m_\theta(\eta)\right\| 
\end{align*}
for every $\eta\in E^u(\theta)$ with $\left\| \eta\right\| =1$.\\
Case 2: $\omega^u\neq 0$.\\
Since the geodesic flow is Anosov we have that
\begin{align*}
\left\| d\phi^m_\theta(\omega^s)\right\|\le C\lambda^m L<C^{-1}\lambda^{-m}L\le L \dfrac{\left\| d\phi^m_\theta(\omega^u)\right\|}{\left\|\omega^u \right\| }.
\end{align*}
Then
\begin{align*}
\left\| d\phi^m_\theta(\omega)\right\|&\le \left\| d\phi^m_\theta(\omega^s)\right\| + \left\| d\phi^m_\theta(\omega^u)\right\|+\left\| d\phi^m_\theta(\omega^c)\right\|\\
&\le L \dfrac{\left\| d\phi^m_\theta(\omega^u)\right\|}{\left\|\omega^u \right\| } +  L \dfrac{\left\| d\phi^m_\theta(\omega^u)\right\|}{\left\|\omega^u \right\| }+1\\
&< (2L+1)\dfrac{\left\| d\phi^m_\theta(\omega^u)\right\|}{\left\|\omega^u \right\| }.
\end{align*}
If we consider $\tau_1=2L+1$, in both cases we have that
\begin{align*}
	\left\| d\phi^m_\theta(\omega)\right\|\le \tau_1 \left\| \left. d \phi^m_\theta\right|_{E^u(\theta)} \right\| 
\end{align*}
for every $\omega\in T_\theta SM$ with $\left\| \omega\right\| =1$. Since the norm is always attained in a finite-dimensional space, we conclude the proof of the lemma.
\end{proof}
\begin{lemma}
For $m\in \mathbb{N}$ large enough, there is $\tau_2\in (0,1)$ such that for every $\theta\in SM$
\begin{align*}
	\left\|d\phi^m_\theta \right\|^*\ge \tau_2 \left\| d\phi^m_\theta(\xi)  \right\|  
\end{align*}
for some $\xi\in E^s(\theta)$ with $\left\| \xi\right\| =1$, where $\displaystyle \left\| d\phi^m_\theta\right\|^*=\inf_{\left\|v \right\| =1}\left\|d\phi^m_\theta(v) \right\|$.
\end{lemma}
\begin{proof} Let $\varepsilon>0$ and consider $m\in \mathbb{N}$ large enough such that $\varepsilon\ge (L+1)C\lambda^m$ and $\sqrt{1-\varepsilon^2} >\varepsilon C\lambda^m$, where $L$ comes from Section 3. Fix $\theta\in SM$ and define the following set
\begin{align*}
\Gamma_{\theta,\varepsilon,m}:=\left\lbrace \omega\in T_\theta SM: \left\| \omega\right\| =1,  \omega=\omega^s+\omega^u+\omega^c \hspace{0.2cm} \text{and} \hspace{0.2cm} \left\| d\phi^m_\theta(\omega^u+\omega^c) \right\| \ge \varepsilon  \right\rbrace .
\end{align*}
Case 1: $\omega\in \Gamma_{\theta,\varepsilon,m}$ with $\omega^s\neq 0$.\\
Since the geodesic flow is Anosov,
\begin{align}\label{gamma1}
	\left\| d\phi^m_\theta(\omega) \right\| 
	&\ge  \left\| d\phi^m_\theta(\omega^u+\omega^c) \right\| -\left\| d\phi^m_\theta(\omega^s) \right\| \nonumber\\
	&\ge\left\| d\phi^m_\theta(\omega^u+\omega^c) \right\| - C\lambda^m \left\| \omega^s\right\|. 
\end{align}
As $\omega\in \Gamma_{\theta,\varepsilon,m}$ we have that
\begin{align}\label{gamma2}
\left\| d\phi^m_\theta(\omega^u+\omega^c) \right\|\ge \varepsilon\ge (L+1)C\lambda^m\ge (\left\| \omega^s\right\| +1)C\lambda^m.
\end{align}
Then from \eqref{gamma1} and \eqref{gamma2}
\begin{align*}
\left\| d\phi^m_\theta(\omega) \right\| \ge \left\| d\phi^m_\theta(\omega^u+\omega^c) \right\|-C\lambda^m \left\| \omega^s\right\| \ge C\lambda^m \ge \dfrac{\left\| d\phi^m_\theta(\omega^s) \right\|}{\left\| \omega^s \right\|}.
\end{align*}
Case 2: $\omega\in \Gamma_{\theta,\varepsilon,m}$ with $\omega^s=0 $.\\
Since the geodesic flow is Anosov,
\begin{align*}
\left\| d\phi^m_\theta(\omega) \right\| = \left\| d\phi^m_\theta(\omega^u+\omega^c) \right\|\ge \varepsilon>C\lambda^m\ge \left\|d\phi^m_\theta(\xi) \right\| 
\end{align*}
for every $\xi\in E^s(\theta)$ with $\left\|\xi \right\| =1$.\\
Case 3: $\omega\notin \Gamma_{\theta,\varepsilon,m}$ and $\left\|w\right\|=1$.\\
We have that
\begin{align*}
\varepsilon^2&>\left\|d\phi^m_\theta(\omega^u+\omega^c) \right\|^2= \left\|d\phi^m_\theta(\omega^u) \right\|^2+\left\|d\phi^m_\theta(\omega^c) \right\|^2=\left\|d\phi^m_\theta(\omega^u) \right\|^2+\left\|\omega^c \right\|^2.
\end{align*}
Then $\left\|\omega^c \right\|<\varepsilon$ and $\left\|d\phi^m_\theta(\omega^u) \right\|<\varepsilon$. Since the geodesic flow is Anosov,
\begin{align*}
C^{-1}\lambda^{-m}\left\| \omega^u\right\| \le \left\| d\phi^m_\theta(\omega^u)\right\|<\varepsilon.  
\end{align*}
This implies that $\left\|\omega^u \right\|<\varepsilon C\lambda^m $. On the other hand, as $\left\|w\right\|=1$, then
\begin{align*}
\left\| \omega^u\right\| + \left\| \omega^s\right\| \ge \left\| \omega^u+\omega^s\right\|=\sqrt{1-\left\| \omega^c\right\|^2}>\sqrt{1-\varepsilon^2 } .
\end{align*}
Furthermore
\begin{align}\label{v_s}
L\ge \left\| \omega^s\right\| > \sqrt{1-\varepsilon^2} - \varepsilon C\lambda^m>0.
\end{align}
In particular, $\omega^s\neq 0$. Denote by
\begin{align*}
	E^{cu}(\theta):=E^u(\theta)\oplus\left\langle G(\theta)\right\rangle
\end{align*}
and define the following linear map
\begin{align*}
P_\theta: T_\theta SM\rightarrow E^s(\theta) 
\end{align*}
as the parallel projection onto $E^s(\theta)$ along $E^{cu}(\theta)$. Since the angle between the stable and unstable subspaces is uniformly away from $0$ for every $\theta\in SM$, then there is $\delta\ge 1$ such that
\begin{align*}
	\left\|P_\theta(\omega) \right\| \le \delta\left\| \omega\right\| 
\end{align*}
for every $\theta\in SM$ and $\omega\in T_\theta SM$ (see Theorem 3.1 in \cite{ipsen}). Then
\begin{align*}
\left\|d\phi^m_\theta(\omega^s) \right\|=\left\|P_{\phi^m(\theta)}(d\phi^m_\theta(\omega)) \right\|  \le \delta\left\|d\phi^m_\theta(\omega) \right\|. 
\end{align*}
By \eqref{v_s}, if we choose $\varepsilon>0$ such that $\left\| \omega^s\right\| \ge 1/2$, we have that
\begin{align*}
\left\| d\phi^m_\theta(\omega)\right\| \ge \dfrac{1}{2\delta}\dfrac{\left\| d\phi^m_\theta(\omega^s)\right\| }{\left\| \omega^s\right\| }.
\end{align*}
If we consider $\tau_2= 1/2\delta$, in all cases we have that for all for every $\omega\in T_\theta SM$ with $\left\| \omega\right\| =1$ there is $\xi \in E^s(\theta)$, with $\left\|\xi\right\|=1$ such that 
$$\left\| d_\theta\phi^m(\omega)\right\| \ge \tau_2 \left\|  d_\theta\phi^m(\xi)\right\|.$$
Since the infimum is always attained in a finite-dimensional space, the last inequality concludes the proof of the lemma.
\end{proof}
Clearly $\left\| d\phi^m_\theta\right\|^* \le \left\| d\phi^m_\theta\right\| $ for every $\theta\in SM$. From Lemmas 4.1 and 4.2, we can obtain a positive constant, independent of $\theta$, such that the direction of the inequality changes.
\begin{proposition}
For $m\in \mathbb{N}$ large enough, there is $\kappa>1$, depending on $m$, such that
\begin{align*}
	\left\| d\phi^m_\theta\right\|\le \kappa\left\| d\phi^m_\theta\right\|^*  
\end{align*}
for every $\theta\in SM$.
\end{proposition}
\begin{proof} From Lemma 4.1 we have that
\begin{align*}
\left\| d \phi^m_\theta\right\|\le \tau_1 \left\| d \phi^m_\theta(\eta)\right\|    
\end{align*}
for some $\eta\in E^u(\theta)$ with $\left\|\eta\right\|=1$. Denote by $J_\eta$ the Jacobi field associated to $\eta$. Since the geodesic flow is of Anosov type and $d\phi^t_\theta(\eta)=(J_\eta(t),J'_\eta(t))$, we have from item 1 of Section 2.3 that
\begin{align}\label{uno}
\left\| d \phi^m_\theta\right\|\le \tau_1 \left\| d \phi^m_\theta(\eta)\right\|=\tau_1 \sqrt{\left\|J_\eta(m) \right\|^2 +\left\|J'_\eta(m) \right\|^2  } \le \tau_1\sqrt{1+c^2}\left\| J_\eta(m) \right\|. 
\end{align}
In the same way, by Lemma 4.2 we have that
\begin{align}\label{dos}
\left\| d\phi^m_\theta\right\|^*\ge \tau_2 \left\| d\phi^m_\theta(\xi)\right\|=\tau_2\sqrt{1+\dfrac{\left\| J'_\xi(m)\right\|^2 }{\left\| J_\xi(m) \right\|^2 }}\left\| J_\xi(m) \right\|,
\end{align}
for some $\xi\in E^s(\theta)$ with $\left\|\xi\right\|=1$, where $J_\xi$ is the Jacobi field associated to $\xi$. Moreover,
\begin{align}\label{tres}
\sqrt{1+c^2}\le \sqrt{1+c^2}\sqrt{1+\dfrac{\left\| J'_\xi(m)\right\|^2 }{\left\| J_\xi(m) \right\|^2 }}.
\end{align}
Define the function
\begin{align*}
r:&[0,+\infty) \rightarrow \hspace{0.4cm}\mathbb{R}\\
&\hspace{0.6cm}t\hspace{0.65cm}\rightarrow \dfrac{\lambda^{-t}\left\| J_\xi(t)\right\| }{\lambda^t \left\| J_\eta(t)\right\| }.
\end{align*}
This function is well-defined because the stable and unstable Jacobi fields are never zero since the manifold has no conjugate points (see Section 2). We have that
\begin{align*}
	r'(t)=r(t)\left(-2\log \lambda + \dfrac{\left\langle J'_\xi(t),J_\xi(t) \right\rangle }{\left\langle J_\xi(t),J_\xi(t) \right\rangle } - \dfrac{\left\langle J'_\eta(t),J_\eta(t) \right\rangle }{\left\langle J_\eta(t),J_\eta(t) \right\rangle } \right). 
\end{align*}
Also
\begin{align*}
A(t)=\dfrac{ \left\langle J'_\xi(t),J_\xi(t) \right\rangle }{\left\langle J_\xi(t),J_\xi(t) \right\rangle }\in [-c,c] \hspace{0.4cm}\text{and} \hspace{0.4cm}B(t)=\dfrac{ \left\langle J'_\eta(t),J_\eta(t) \right\rangle }{\left\langle J_\eta(t),J_\eta(t) \right\rangle }\in [-c,c].
\end{align*}
Since the curvature is bounded below by $-c^2$, then $\lambda\ge e^{-c}$ (see \cite{rigidity}). Therefore
\begin{align*}
-2\log \lambda-2c\le -2\log \lambda +A(t) - B(t)\le -2\log \lambda + 2c.
\end{align*}
This implies that 
\begin{align*}
-2\log\lambda - 2c\le \dfrac{r'(t)}{r(t)} \le -2\log \lambda + 2c
\end{align*}
and
\begin{align*}
r(0)\cdot e^{(-2\log \lambda-2c)t}\le r(t)\le r(0)\cdot e^{(-2\log \lambda+2c)t}.
\end{align*}
Therefore
\begin{align*}
r(0)^{-1}\cdot e^{(2\log\lambda - 2c)t}\le \dfrac{1}{r(t)}\le r(0)^{-1}\cdot e^{(2\log\lambda+2c)t}.
\end{align*}
For $t=m$ we have that
\begin{align}\label{cuatro}
\left\|J_\eta(m) \right\|\le r(0)^{-1}\cdot e^{(2\log\lambda+2c)m}\cdot\lambda^{-2m}\left\| J_\xi(m) \right\|=r(0)^{-1}\cdot e^{2cm}\left\| J_\xi(m) \right\|.  
\end{align}
From \eqref{uno}, \eqref{dos}, \eqref{tres} and \eqref{cuatro}
\begin{align}\label{cotaK}
	\left\| d\phi^m_\theta\right\|&\le \tau_1 \sqrt{1+c^2}\left\|J_\eta(m) \right\| \nonumber\\
	&\le  \tau_1\sqrt{1+c^2}\cdot r(0)^{-1}\cdot e^{2cm}\left\| J_\xi(m) \right\| \nonumber \\
	&  \le \tau_1\sqrt{1+c^2}\sqrt{1+\dfrac{\left\| J'_\xi(m)\right\|^2 }{\left\| J_\xi(m) \right\|^2 }}\cdot r(0)^{-1}\cdot e^{2cm}\left\| J_\xi(m) \right\|.
\end{align}
From item 1 of Section 2.3 we have that
\begin{align*}
	1=\left\|\xi \right\|^2=\left\|d\pi_\theta(\xi) \right\|^2 + \left\| K_\theta(\xi)\right\|^2\le (1+c^2)  \left\|d\pi_\theta(\xi) \right\|^2.
\end{align*}
Since $1=\left\|\eta \right\|^2=\left\|d\pi_\theta(\eta) \right\|^2 + \left\| K_\theta(\eta)\right\|^2$, the last inequality implies that
\begin{align*}
r(0)^{-1}=\dfrac{\left\|d\pi_\theta(\eta) \right\|}{\left\|d\pi_\theta(\xi) \right\|}\le\dfrac{1}{\left\|d\pi_\theta(\xi) \right\|}\le  \sqrt{1+c^2}.
\end{align*}
Therefore, substituting in \eqref{cotaK} and using \eqref{dos}
\begin{align*}
	\left\| d\phi^m_\theta\right\|\le \kappa \left\| d\phi^m_\theta\right\|^*,
\end{align*}
where $\kappa=\tau_1\cdot\tau_2^{-1}\cdot (1+c^2)\cdot e^{2cm}>1$. 
\end{proof}
On the other hand, since the geodesic flow is of Anosov type, we have that the norm $\left\| d\phi^m_\theta\right\|$ is bounded between two positive constants.
\begin{proposition}
For $m\in \mathbb{N}$ large enough, there are constants $K_1, K_2>0$, $K_1$ depending on $m$, such that
\begin{align*}
	K_2<\left\| d\phi^m_\theta\right\| <K_1
\end{align*}
for every $\theta\in SM$.
\end{proposition}
\begin{proof} Fix $\theta\in SM$. Since the geodesic flow is of Anosov type, for $\eta\in E^u(\theta)$ with $\left\|\eta \right\| =1$ we have that
\begin{align*}
	\left\| d\phi^m_\theta\right\|\ge \left\| d\phi^m_\theta(\eta)\right\|\ge C^{-1}\lambda^{-m}>C^{-1},
\end{align*}
then $K_2=1/C$. On the other hand, from \eqref{cotasup} we have that
\begin{align}\label{betanov}
\left\|d\phi^1_\theta\right\| &\le LC\lambda + L\sqrt{1+c^2}\left( \left( \dfrac{1+c}{c}\right) \sinh c+C\lambda\sqrt{1+c^2}\right) +1\nonumber \\
	&\le LC\lambda + L\sqrt{1+c^2}\left( \dfrac{1+c}{c}\right) \sinh c + LC\lambda(1+c^2)+1\nonumber \\
	&\le 2LC\lambda + LC\lambda c^2 +L\sqrt{1+c^2}\left( \dfrac{1+c}{c}\right) \sinh c + 1:=h(c).
\end{align}
Then, we can consider $K_1=h(c)^m$.
\end{proof}
A direct consequence of Proposition 4.4 is the following result.
\begin{corollary}
Given $\varepsilon>0$, there is $\beta\in (0,1)$, depending on $m$, such that
\begin{align*}
	\beta \left\| d\phi^m_{\tilde{\theta}}\right\|< \left\| d\phi^m_\theta\right\|, \hspace{0.4cm}\forall\hspace{0.1cm} \tilde{\theta}\in SM: d(\theta,\tilde{\theta})<\varepsilon
\end{align*}
for every $\theta\in SM$.
\end{corollary}
\begin{proof}
By Proposition 4.4 we have that
\begin{align*}
	\dfrac{K_2}{K_1}<\dfrac{\left\| d\phi^m_\theta\right\|}{\left\| d\phi^m_{\tilde{\theta}}\right\|} <\dfrac{K_1}{K_2}
\end{align*}
Considering $\beta=\dfrac{K_2}{K_1}=\dfrac{C^{-1}}{h(c)^m}$ the conclusion of the corollary follows. 
\end{proof} 
\section{Ruelle's Inequality}
In this section, we will prove Theorem 1.1. For this, we will adapt the idea of the proof of Ruelle's inequality for diffeomorphisms in the compact case exhibited in \cite{pesin}. \\

Let $M$ be a complete Riemannian manifold satisfying all the hypotheses of Theorem 1.1 and $\mu$ an $\phi^t$-invariant probability measure on $SM$. By simplicity, we consider $\mu$ an ergodic $\phi^t$-invariant probability measure on $SM$. In this case, we denote by $\left\lbrace \mathcal{X}_i\right\rbrace$ the Lyapunov exponents and $\left\lbrace k_i\right\rbrace$ their respective multiplicities. The proof in the non-ergodic case is a consequence of the ergodic decomposition of such a measure. We can also assume that $\phi=\phi^1$ is an ergodic transformation with respect to $\mu$. If it is not the case, we can choose an ergodic-time $\tau$ for $\mu$ and prove the theorem for the map $\phi^{\tau}$. The proof of the theorem for the map $\phi^\tau$ implies the proof for the map $\phi$ because the entropy of $\phi^\tau$ and the Lyapunov exponents are $\tau$-multiples of the respective values of $\phi$.\\

Fix $\varepsilon>0$ and $m\in \mathbb{N}$ large enough. There exists a compact set $K\subset SM$ such that $\mu(K)>1-\varepsilon$.
Based on the results in Section 4, we present the following theorem, which constitutes a similar version to the inclusion (10.3) described in \cite{pesin}. 
Consider the constants $\kappa>1$ and $0<\beta<1$ given by Proposition 4.3 and Corollary 4.5 respectively.
\begin{theorem}
Let $M$ be a complete Riemannian manifold without conjugate points and sectional curvature bounded below by $-c^2$, for some $c>0$. If the geodesic flow is of Anosov type, then for every $\theta\in K$ there exists $\varrho:=\varrho(K)\in (0,1)$ such that
\begin{align*}
\phi^m(exp_\theta(B(0,\beta \kappa^{-1}\varrho)))\subseteq exp_{\phi^m(\theta)} (d \phi^m_\theta(B(0,\varrho)).
\end{align*}
\end{theorem}
\noindent
\begin{proof} We will proceed by contradiction. Suppose that for every $n\in \mathbb{N}$, there are $\theta_n\in K$ and $v_n\in T_{\theta_n}SM$ with $\left\| v_n\right\|=\dfrac{\beta \kappa^{-1}}{n}$ such that
\begin{align*}
	\phi^m(exp_{\theta_n}(v_n))=exp_{\phi^m(\theta_n)}(d\phi^m_{\theta_n}(w_n))
\end{align*}
where $\left\|w_n \right\|=\dfrac{1}{n}$. Since $K$ is compact and $w_n\rightarrow 0$, then $\left\|d\phi^m_{\theta_n}(w_n) \right\| $ is less than injectivity radius of the
exponential map restricted to the compact set $K$, for $n$ large enough by Proposition 4.4. Therefore
\begin{align*}
	\left\| d\phi^m_{\theta_n}(w_n)\right\| &=d(\phi^m(\theta_n),exp_{\phi^m(\theta_n)}(d\phi^m_{\theta_n}(w_n)))\\
&=d(\phi^m(\theta_n),\phi^m(exp_{\theta_n}(v_n)))\\
	&\leq \int_0^1\left\|(\phi^m\circ c_n)'(t) \right\| dt ,
\end{align*} 
where $c_n(t)=exp_{\theta_n}(tv_n)$. Then
\begin{align*}
	\left\| d\phi^m_{\theta_n}(w_n)\right\|&\le \sup_{t\in [0,1]}\left\|d\phi^m_{c_n(t)} \right\| \int_0^1 \left\| c'_n(t) \right\| dt\\
	&=\sup_{t\in [0,1]}\left\|d\phi^m_{c_n(t)} \right\|\cdot \left\| v_n\right\| .
\end{align*}
For $n$ large enough, by Corollary 4.5 we have that
\begin{align*}
\kappa\beta^{-1}\dfrac{\left\| d\phi^m_{\theta_n}(w_n)\right\|}{\left\| w_n\right\| }& = \dfrac{\left\| w_n\right\| }{\left\| v_n\right\| }\dfrac{\left\| d\phi^m_{\theta_n}(w_n)\right\|}{\left\| w_n\right\| }\\
	&\le \sup_{t\in [0,1]}\left\|d_{c_n(t)}\phi^m \right\|\\
	&<\beta^{-1} \left\|d\phi^m_{\theta_n} \right\|.
\end{align*}
Therefore
\begin{align*}
\kappa\beta^{-1}\left\|d\phi^m_{\theta_n} \right\|^*<\beta^{-1} \left\|d\phi^m_{\theta_n} \right\|
\end{align*}
which contradicts the Proposition 4.3. 
\end{proof}
Now, denote by $\varrho_m=\beta \kappa^{-1}\varrho<1$, where the constants $\beta, \kappa$ and $\varrho$ come from Theorem 5.1.  Using the techniques of separate sets applied in \cite{pesin} we define a finite partition $\mathcal{P}=\mathcal{P}_{K}\cup \left\lbrace SM\setminus K\right\rbrace $ of $SM$ in the following way:
\begin{itemize}
	\item[.] $\mathcal{P}_{K}$ is a partition of $K$ such that for every $X\in \mathcal{P}_{K}$, there exist balls $B(x,r')$ and $B(x,r)$ such that the constants satisfy $0<r'<r<2r'\le   \dfrac{\varrho_m}{2}$ and $$B(x,r')\subset X\subset B(x,r).$$
	\item[.] There exists a constant $\zeta>0$ such that the cardinal of $\mathcal{P}_{K}$, denoted by $\left| \mathcal{P}_{K}\right|$, satisfies $$\left| \mathcal{P}_{K}\right|\le \zeta\cdot(\varrho_m)^{-\dim(SM)}.$$
	\item[.] $h_{\mu}(\phi^m,\mathcal{P})\ge h_{\mu}(\phi^m)-\varepsilon$.
\end{itemize}
By definition of entropy,
\begin{align}\label{ruelle}
	h_{\mu}(\phi^m,\mathcal{P})&=\lim_{k\to +\infty} H_{\mu}\left( \left. \mathcal{P}\right|\phi^m\mathcal{P} \vee\ldots \vee \phi^{km}\mathcal{P} \right) \nonumber \\
	&\le H_{\mu}\left( \left. \mathcal{P}\right| \phi^m\mathcal{P}\right) \nonumber \\
	&\le \sum_{D\in \phi^m\mathcal{P}}\mu(D)\cdot\log\text{card}\left\lbrace X\in \mathcal{P}: X\cap D\neq\emptyset \right\rbrace.
\end{align}
Denote by $\varphi=\sup_{\theta\in SM}\left\|d\phi_\theta \right\|>1$. First, we estimate the number of elements $X\in \mathcal{P}$ that intersect a given element $D\in \phi^m\mathcal{P}$.
\begin{lemma}
	There exists a constant $L_1>0$ such that if $D\in \phi^m\mathcal{P}$ then
	\begin{align*}
		\emph{card} \left\lbrace X\in \mathcal{P}: X\cap D\neq\emptyset \right\rbrace \le L_1\cdot\max\left\lbrace \varphi^{m\cdot\dim (SM)}, (\varrho_m)^{-\dim(SM)} \right\rbrace .
	\end{align*}
\end{lemma}
\begin{proof}
Consider $D\in \phi^m\mathcal{P}$, then $D=\phi^m(X')$ for some $X'\in \mathcal{P}$.\\
Case I: $X'\in \mathcal{P}_{K}$.\\
By the mean value inequality
\begin{align*}
	\text{diam}(D)&=\text{diam}(\phi^m(X'))\\
	&\le \sup_{\theta\in SM} \left\|d\phi_\theta \right\|^{m}\cdot\text{diam}(X')\\
	&\le \varphi^m\cdot 4r',
\end{align*}
since $X'\subset B(x,2r')$. If $X\in \mathcal{P}_{K}$ satisfies $X\cap D\neq \emptyset$, then $X$ is contained in a $4r'$-neighborhood of $D$, denoted by $W$. Since $\varphi^m>1$ we have that
\begin{align*}
	\text{diam}(W)& \le \varphi^m\cdot 4r' + 8r'\\
	&=4r'\cdot \left( \varphi^m + 2\right)\\
 &<12r'\cdot\varphi^m.
\end{align*} 
Hence
\begin{align}\label{vol1}
	\sum_{\left\lbrace X\in \mathcal{P}_{K}:X\cap D\neq \emptyset\right\rbrace }\text{vol}(X)\le \text{vol}(W)\le A_1 \cdot (r')^{\dim (SM)} \cdot\varphi^{m\cdot\dim (SM)},
\end{align}
where $A_1>0$. Since $X\in \mathcal{P}_{K}$ contains a ball of radius $r'$, the volume of $X$ is bounded below by
\begin{align}\label{vol2}
	A_2\cdot (r')^{\dim (SM)}\le\text{vol}(X),
\end{align}
where $A_2>0$. From \eqref{vol1} and \eqref{vol2} we have that
\begin{align*}
	\text{card}\left\lbrace X\in \mathcal{P}:X\cap D\neq \emptyset\right\rbrace &\le \dfrac{A_1}{A_2} \cdot\varphi^{m\cdot\dim (SM)} +1\\
	&\le \left(\dfrac{A_1}{A_2}+1\right) \cdot\varphi^{m\cdot\dim (SM)} .
\end{align*}
\noindent
Case II: $X'=SM\setminus K$.\\
In this case, we have that 
\begin{align*}
	\text{card}\left\lbrace X\in \mathcal{P}:X\cap D\neq \emptyset\right\rbrace&\le \left| \mathcal{P}_{K}\right| +1\\
	&\le (\zeta+1)(\varrho_m)^{-\dim(SM)}.
\end{align*}
Considering $L_1=\max\left\lbrace \dfrac{A_1}{A_2}+1, \zeta+1\right\rbrace$ we obtain the desired result.
\end{proof}
Now we will get a finer exponential bound for the number of those sets $D\in \phi^m\mathcal{P}_K$ that contain regular points. For this, let $\Lambda_m$ be the set of regular points $\theta\in SM$ which satisfy the following condition: for $k\ge m$ and $\xi\in T_\theta SM$
\begin{align*}
	e^{k\left( \mathcal{X}(\theta,\xi)-\varepsilon\right) } \left\| \xi\right\| \le \left\| d\phi^k_\theta (\xi) \right\|\le e^{k\left( \mathcal{X}(\theta,\xi)+\varepsilon\right) } \left\| \xi\right\|,
\end{align*}
where $	\mathcal{X}(\theta,\xi)=\displaystyle \lim_{n\rightarrow \pm \infty} \dfrac{1}{n}\log \left\|d \phi^n_\theta(\xi) \right\|$. 
\begin{lemma}
If $D\in \phi^m\mathcal{P}_K$ has non-empty intersection with $\Lambda_m$, then there is a constant $L_2>0$ such that
\begin{align*}
	\emph{card}\left\lbrace X\in \mathcal{P}: X\cap D\neq \emptyset\right\rbrace\le L_2\cdot e^{m\varepsilon}\prod_{i:\mathcal{X}_i>0}e^{m(\mathcal{X}_i+\varepsilon)k_i}. 
\end{align*}
\end{lemma}
\begin{proof}
 Let $X'\in \mathcal{P}_{K}$ such that $\phi^m(X')=D$ and suppose that $X'\cap \Lambda_m\neq \emptyset$. Pick a point $\theta\in X'\cap \Lambda_m$ and consider the ball $B=B(0,\varrho)\subset T_\theta SM$. We claim that $$X'\subseteq exp_\theta(B(0,\varrho_m)),$$
where $exp_\theta$ denotes the exponential map defined on the tangent plane $T_\theta SM$. In fact, let $z\in X'$. Since $SM$ is complete with the Sasaki metric (see Lemma 2.1) we can choose $w\in T_\theta SM$ such that $\gamma(t)=exp_\theta(tw)$, where $\gamma$ is a geodesic with $\gamma(0)=\theta$ and $\gamma(1)=exp_\theta(w)=z$. As diam $\mathcal{P}_{K} < \varrho_m$ then $$d(\theta,z)=l(\gamma)< \varrho_m.$$
Similar to the proof of Proposition 2.2, we obtain that
\begin{align*}
	\varrho_m> \int_0^1\left\| \gamma'(s)\right\| ds
	=\left\| w\right\| .
\end{align*}
Then $w\in B(0,\varrho_m)$ and hence
$$z=exp_{\theta}(w)\in exp_\theta(B(0,\varrho_m)).$$
Since $z\in X'$ was arbitrary, the claim is proven. Therefore, from Theorem 5.1 we have that
\begin{align*}
	D=\phi^m(X')\subseteq B_0:=exp_{\phi^m(\theta)}(\tilde{B}_0),
\end{align*}
where $\tilde{B}_0=d\phi^m_\theta(B)$ is an ellipsoid. Since the curvature tensor and the derivative of the curvature tensor of $M$ are both uniformly bounded, we have that the Sasaki sectional curvature of $SM$ is uniformly bounded (see \eqref{curvaturaSM}). This implies that the curvature tensor of $SM$ is uniformly bounded. Applying Proposition 2.2 to $SM$, there exists $t_0>0$ such that  
\begin{align}\label{novot0}
\left\|d (exp_{\phi^m(\theta)})_{tv} \right\|\le \dfrac{5}{2}    
\end{align} 
for every $\left| t\right|\le t_0$ and $v\in T_{\phi^m(\theta)}SM$ with $\left\|v \right\|=1$. Then, for $m$ large enough, we have that
\begin{align*}
	\text{diam} (D)&\le h(c)^m\cdot \text{diam} (X')\\
	&\le h(c)^m\cdot \varrho_m\\
	&=\dfrac{1}{C}\cdot\dfrac{\tau_2}{\tau_1}\cdot\dfrac{1}{1+c^2}\cdot e^{-2cm}\cdot \varrho\\
	&<\dfrac{t_0}{2}, 
\end{align*}
where $h(c)$ is the expression that bounds the derivative of  $\phi$ (see Proposition 4.4). Therefore, we can choose $B_0$ that satisfies $D\subset B_0$ and diam$(B_0)<t_0$. We know that diam$\mathcal{P}_{K}< \varrho_m<\varrho$, then if $X\in \mathcal{P}_{K}$ intersects $D$, it lies in the set
$$B_1=\left\lbrace \Psi\in SM: d(\Psi,B_0)<\varrho\right\rbrace.$$ 
Since $X\subset B(x,r)$ and $2r<\varrho_m<\varrho$, then $B(x,\varrho/2)\subset B_1$ and 
\begin{align}\label{cardi}
	\text{card}\left\lbrace X\in \mathcal{P}_{K}: X\cap D\neq \emptyset \right\rbrace \le b\cdot\text{vol}(B_1)\cdot \varrho^{-\dim(SM)},
\end{align}
for some $b>0$, where vol$(B_1)$ denotes the volume of $B_1$ induced by the Sasaki metric. Consider a subset $\tilde{B}^*_0\subset \tilde{B}_0$ such that $exp_{\phi^m(\theta)}$ is a diffeomorphism between $\tilde{B}^*_0$ and $B_0$. Since
\begin{align*}
	\left| \det d (exp_{\phi^m(\theta)})_v\right| \le \left\|d (exp_{\phi^m(\theta)})_v \right\|^{\dim(SM)}
\end{align*}
for every $v\in \tilde{B}^*_0$, from \eqref{novot0} we have that
\begin{align*}
	\text{vol}(B_0)\le \left( \dfrac{5}{2}\right) ^{\dim(SM)}\cdot \text{vol}(\tilde{B}_0).
\end{align*}
This implies that the volume of $B_1$ is bounded, up to a bounded factor, by the product of the lengths of the axes of the ellipsoid $\tilde{B}_0$. Those corresponding to non-positive Lyapunov exponents are at most sub-exponentially large. The remaining ones are of size at most $e^{m(\mathcal{X}_i+\varepsilon)}$, up to a bounded factor, for all sufficiently large $m$. Thus 
\begin{align*}
	\text{vol}(B_1)&\le A\cdot e^{m\varepsilon}\cdot(\text{diam}(B))^{\dim (SM)}\prod_{i:\mathcal{X}_i>0}e^{m(\mathcal{X}_i+\varepsilon)k_i}\\
	&\le A\cdot e^{m\varepsilon}\cdot(2\varrho)^{\dim (SM)}\prod_{i:\mathcal{X}_i>0}e^{m(\mathcal{X}_i+\varepsilon)k_i}\\
	& =\tilde{A}\cdot e^{m\varepsilon}\cdot \varrho^{\dim (SM)}\prod_{i:\mathcal{X}_i>0}e^{m(\mathcal{X}_i+\varepsilon)k_i},
\end{align*}
where $\tilde{A}=A\cdot 2^{\dim (SM)}$, for some $A>0$. Then substituting in \eqref{cardi} we have that
\begin{align*}
	\text{card}\left\lbrace X\in \mathcal{P}: X\cap D\neq \emptyset \right\rbrace &\le b\cdot\text{vol}(B_1)\cdot \varrho^{-\dim (SM)}+1\\
	&\le b\cdot\tilde{A}\cdot e^{m\varepsilon}\prod_{i:\mathcal{X}_i>0}e^{m(\mathcal{X}_i+\varepsilon)k_i}+1\\
	&\le (b\cdot\tilde{A}+1)\cdot e^{m\varepsilon}\prod_{i:\mathcal{X}_i>0}e^{m(\mathcal{X}_i+\varepsilon)k_i}.
\end{align*}
Considering $L_2=b\cdot\tilde{A}+1$ we obtain the desired result.
\end{proof}
\textbf{Proof of Theorem 1.1.} We have that $\mu(SM\setminus K)<\varepsilon$. From \eqref{ruelle}, Lemmas 5.2 and 5.3 we obtain
\begin{align*}	
	mh_{\mu}(\phi)-\varepsilon&=h_{\mu}(\phi^m)-\varepsilon \nonumber\\
	&\le h_{\mu}(\phi^m,\mathcal{P}) \nonumber\\
	&\le \sum_{D\in \phi^m\mathcal{P}}\mu(D)\cdot\log\text{card}\left\lbrace X\in \mathcal{P}: X\cap D\neq\emptyset \right\rbrace \nonumber\\
	&\le \sum_{D\in \phi^m\mathcal{P}_K, D\cap \Lambda_m=\emptyset}\mu(D)\cdot\log\text{card}\left\lbrace X\in \mathcal{P}: X\cap D\neq\emptyset \right\rbrace \nonumber\\
	&\hspace{0.5cm} +	\sum_{D\in \phi^m\mathcal{P}_K, D\cap \Lambda_m\neq \emptyset}\mu(D)\cdot\log\text{card}\left\lbrace X\in \mathcal{P}: X\cap D\neq\emptyset \right\rbrace \nonumber\\
&\hspace{0.5cm}+\mu(\phi^m(SM\setminus K))\cdot\log\text{card}\left\lbrace X\in \mathcal{P}: X\cap \phi^m(SM\setminus K)\neq\emptyset \right\rbrace \nonumber\\
&\le \sum_{D\in \phi^m\mathcal{P}_K, D\cap \Lambda_m=\emptyset}\mu(D)\left( \log (L_1) +\dim (SM)\cdot\max\left\lbrace m\log(\varphi),-\log(\varrho_m) \right\rbrace\right) \nonumber\\
	&\hspace{0.5cm}+ \sum_{D\in \phi^m\mathcal{P}_K, D\cap \Lambda_m\neq \emptyset}\mu(D)\left( \log(L_2) +m\varepsilon  +m\sum_{i:\mathcal{X}_i>0}(\mathcal{X}_i+\varepsilon)k_i\right) \nonumber \\
	& \hspace{0.5cm} + \mu(SM\setminus K)\cdot\left( \log (L_1) +\dim (SM)\cdot\max\left\lbrace m\log(\varphi),-\log(\varrho_m) \right\rbrace\right) \nonumber
 \end{align*}
 \begin{align}\label{fin}
\hspace{1.2cm}
	&\le \left( \log (L_1) +\dim (SM)\cdot\max\left\lbrace m\log\left(\varphi\right) ,-\log(\varrho_m) \right\rbrace\right)\cdot \mu(SM\setminus \Lambda_m) \nonumber\\
	&\hspace{0.5cm} + \log(L_2) +m\varepsilon  +m\sum_{i:\mathcal{X}_i>0}(\mathcal{X}_i+\varepsilon)k_i \nonumber\\
	&\hspace{0.5cm}+ \varepsilon\cdot\left( \log (L_1) +\dim (SM)\cdot\max\left\lbrace m\log\left(\varphi\right) ,-\log(\varrho_m) \right\rbrace\right).
\end{align}
By Oseledec's Theorem we have that $\mu(SM\setminus \Lambda_m)\rightarrow 0$ as $m\rightarrow \infty$. Moreover,
\begin{align*}
	\lim_{m\rightarrow +\infty}\dfrac{1}{m}\log(\varrho_m) = -\log(h(c))-2c,
\end{align*}
where $h(c)$ is the expression that bounds the derivative of $\phi$ (see Proposition 4.4). Then, dividing by $m$ in \eqref{fin} and taking $m\rightarrow +\infty$ we obtain
\begin{align*}
	h_{\mu}(\phi)\le \varepsilon + \sum_{i:\mathcal{X}_i>0}(\mathcal{X}_i+\varepsilon)k_i +\varepsilon\cdot\dim (SM)\cdot\max\left\lbrace \log\left(\varphi\right) ,\log(h(c))+2c \right\rbrace.
\end{align*}
Letting $\varepsilon\rightarrow 0$ we have
\begin{align*}
	h_{\mu}(\phi)\le  \sum_{i:\mathcal{X}_i>0}\mathcal{X}_ik_i,
\end{align*}
which is the desired upper bound. $\hfill\square$

\section{Pesin's Formula}
\noindent
In this section, we aim to prove Theorem 1.2. To achieve this goal, we will use the techniques applied by Mañé in \cite{mane} which don't use the theory of stable manifolds. Adopting this strategy greatly simplifies our proof since we only need to corroborate that all the technical hypotheses used by Mañé continue to be satisfied under the condition of the geodesic flow being Anosov. To simplify notation, we write
\begin{align*}
\mathcal{X}^+(\theta)=\sum_{\mathcal{X}_i(\theta)>0}\mathcal{X}_i(\theta)\cdot\dim (H_i(\theta)).
\end{align*}
We start introducing some notations. Set $g:SM\rightarrow SM$ a map and $\rho:SM\rightarrow (0,1)$ a function. For $\theta\in SM$ and $n\ge 0$, define
\begin{align*}
S_n(g,\rho,\theta)=\left\lbrace \omega\in SM: d(g^j(\theta),g^j(\omega))\le \rho(g^j(\theta)), 0\le j\le n\right\rbrace .
\end{align*}
If $\mu$ is a measure on $SM$ and $g$ and $\rho$ are measurable, define
\begin{align*}
h_\mu(g,\rho,\theta)=\limsup_{n\rightarrow\infty}-\dfrac{1}{n}\log\mu(S_n(g,\rho,\theta)).
\end{align*}
Let $E$ be a normed space and $E=E_1\oplus E_2$ a splitting.
We say that a subset $W\subset E$ is a $(E_1,E_2)$-graph if there exists an open set $U\subset E_2$ and a $C^1$-map $\psi:U\rightarrow E_1$ such that $W=\left\lbrace (\psi (x),x): x\in U\right\rbrace$. The number $$\sup\left\lbrace \dfrac{\left\|\psi(x)-\psi(y) \right\|}{\left\| x-y\right\| }: x,y\in U, x\neq y  \right\rbrace$$ is called the dispersion of $W$.\\

Let $M$ be a complete Riemannian manifold and $\mu$ an $\phi^t$-invariant probability measure on $SM$ satisfying the assumptions of Theorem 1.2. Denote by $\nu$ the Lebesgue measure on $SM$. Since the geodesic flow is of Anosov type, consider	$$E^{cs}(\theta)=\left\langle G(\theta)\right\rangle \oplus E^s(\theta)$$ for every $\theta\in SM$. From Theorem 3.1 there is a set $\Lambda\subset SM$ such that $\mu(SM\setminus\Lambda)=0$ and the Lyapunov exponents of $\phi$ exist for every $\theta\in \Lambda$.  Fix any $\varepsilon>0$. By Egorov's and Oseledec's Theorems, there is a compact set $K\subset \Lambda$ with $\mu(K)\ge 1-\varepsilon$ such that the splitting $T_\theta SM=E^{cs}(\theta)\oplus E^u(\theta)$ is continuous when $\theta$ varies in $K$ and, for some $N>0$, there are constants $\alpha>\beta>1$ such that, if $g=\phi^N$, the inequalities
\begin{align}\label{pesin1}
\left\|d g^n_\theta(\eta) \right\|&\ge \alpha^n\left\| \eta\right\| \nonumber \\ 
 \left\|\left.d g^n_\theta\right|_{E^{cs}(\theta)}  \right\|&\le \beta^n \nonumber\\ 
 \log\left| \det\left( \left. d g^n_\theta\right|_{E^u(\theta)} \right) \right| &\ge Nn\left(\mathcal{X}^+(\theta)-\varepsilon\right)
\end{align}
hold for all $\theta\in K$, $n\ge 0$ and $\eta\in E^u(\theta)$.\\

In the same way as in \cite{mane}, in the remainder of this section, we will treat $SM$ as if it were an Euclidean space. The arguments we use can be formalized without any difficulty by the direct use of local coordinates. Since the geodesic flow is $C^1$-Hölder, we have the following result proved by Mañe in \cite{mane}.
\begin{lemma}
For every $\sigma>0$ there is $\xi>0$ such that, if $\theta\in K$ and $g^m(\theta)\in K$ for some $m>0$, then if a set $W\subset SM$ is contained in the ball $B_{\xi^m}(\theta)$ and is a $(E^{cs}(\theta),E^u(\theta))$-graph with dispersion $\le \sigma$, then $g^m(W)$ is a $(E^{cs}(g^m(\theta)),E^u(g^m(\theta)))$-graph with dispersion $\le \sigma$.
\end{lemma}
Fix the constant $\sigma>0$ of the statement of Lemma 6.1 small enough such that exists $a\in (0,1)$, $a\le t_0/2$, where $t_0$ comes from Proposition 2.2 applied to $SM$, with the following property: if $\theta\in K$, $\omega\in SM$ and $d(\theta,\omega)<a$, then for every subspace $E\subset T_\omega SM$ which is a $(E^{cs}(\theta), E^u(\theta))$-graph with dispersion $\le \sigma$ we have
\begin{align}\label{determinante}
\left| \log \left|\det \left( \left. d g_\omega\right|_{E}\right) \right|-\log|\det ( \left. d g_\theta\right|_{E^u(\theta)} )  |  \right| \le \varepsilon.
\end{align}
We proved in Theorem 3.1 that the norm of the derivative of $\phi$ is bounded, then denote $$P=\sup\left\lbrace \log \left|\det \left. \left( d \phi_\theta\right|_E\right)   \right| : \theta\in SM, E\subset T_\theta SM\right\rbrace.  $$
The following proposition is an adaptation of Mañe's result in \cite{mane} applied to the case of Anosov geodesic flow for non-compact manifolds. To ensure a comprehensive understanding of our arguments, we chose to include the full proof provided by Mañé. 
\begin{proposition}
For every small $\varepsilon>0$, there exist a function $\rho:SM\rightarrow (0,1)$ with $\log \rho\in L^1(SM,\mu)$, an integer $N>0$ and a compact set $K'\subset SM$ with $\mu(SM\setminus K')\le 2\sqrt{\varepsilon}$ such that
\begin{align*}
h_\nu(\phi^N,\rho,\theta)\ge N\left( \mathcal{X}^+(\theta)-\varepsilon-\dfrac{\varepsilon}{N}-4P\sqrt{\varepsilon}\right)
\end{align*}
for every $\theta\in K'$.
\end{proposition}
\begin{proof} For $\theta\in K$, define $L(\theta)$ as the minimum integer $\ge 1$ such that $g^{L(\theta)}(\theta)\in K$. This function is well defined for $\mu$-almost every $\theta\in K$ and it is integrable. Extend $L$ to $SM$, putting $L(\theta)=0$ when $\theta\notin K$ and at points of $K$ that do not return to this set. Define $\rho:SM\rightarrow (0,1)$ as
\begin{align}\label{rho}
\rho(\theta)=\min\left\lbrace a, \xi^{L(\theta)}\right\rbrace,
\end{align}
where $a\in (0,t_0/2)$ comes from property \eqref{determinante} and $\xi>0$ comes from Lemma 6.1. Since $L$ is integrable then clearly $\log\rho$ is also integrable. 
On the other hand, by Birkhoff's ergodic theorem, the function
\begin{align*}
\Psi(\theta)=\lim_{n\rightarrow +\infty}\dfrac{1}{n}\text{card}\left\lbrace 0\le j< n: g^j(\theta)\in \Lambda\setminus K\right\rbrace 
\end{align*}
is defined for $\mu$-almost every $\theta\in \Lambda$. Then
\begin{align*}
\varepsilon\ge \mu(\Lambda\setminus K)&=\int_{\Lambda} \Psi d\mu\\
&\ge \int_{\left\lbrace \theta\in \Lambda: \Psi(\theta)>\sqrt{\varepsilon}\right\rbrace }\Psi d\mu \\
&>\sqrt{\varepsilon}\cdot\mu\left( \left\lbrace \theta\in \Lambda: \Psi(\theta)>\sqrt{\varepsilon}\right\rbrace \right).
\end{align*}
Therefore,
\begin{align*}
\mu\left( \left\lbrace \theta\in \Lambda: \Psi(\theta)\le\sqrt{\varepsilon}\right\rbrace \right)\ge 1-\sqrt{\varepsilon}.
\end{align*}
By Egorov's Theorem, there exists a compact set $K'\subset K$ with $\mu(K')\ge 1-2\sqrt{\varepsilon}$ and $N_0>0$ such that, if $n\ge N_0$,
\begin{align}\label{17}
\text{card}\left\lbrace 0\le j< n: g^j(\theta)\in \Lambda\setminus K\right\rbrace \le 2n\sqrt{\varepsilon}
\end{align}
for all $\theta\in K'$. Since the subspaces $E^{cs}(\theta)$ and $E^u(\theta)$ are not necessary orthogonal, there exists $B>0$ such that
\begin{align}\label{vol}
\nu(S_n(g,\rho,\theta))\le B\int_{E^{cs}(\theta)} \nu\left(\left( \omega+E^u(\theta)\right) \cap S_n(g,\rho,\theta) \right) d\nu(\omega)
\end{align}
for every $\theta\in K'$ and $n\ge 0$, where $\nu$ also denotes the Lebesgue measure in the subspaces $E^{cs}(\theta)$ and $\omega+E^u(\theta)$. For $\omega\in E^{cs}(\theta)$, denote by
$$\Omega_n(\omega)=\left(\omega+E^u(\theta)\right) \cap S_n(g,\rho,\theta).$$
Take $D>0$ such that $D>$ vol$(W)$ for every $(E^{cs}(\theta),E^u(\theta))$-graph $W$ with dispersion $\le \sigma$ contained in $B_{\rho(\theta)}(\theta)$, where $\theta\in K'$ and $\rho$ is the function defined in \eqref{rho}. This constant exists because the domain of the graphs is contained in a ball of radius $<1$ and the derivatives of the functions defining the graphs are uniformly bounded in norm by $\sigma$. 
If $g^n(\theta)\in K'$ and $\omega\in E^{cs}(\theta)$, from Lemma $5$ of \cite{mane} we have that $g^n(\Omega_n(\omega))$ is a $(E^{cs}(g^n(\theta)),E^u(g^n(\theta)))$-graph with dispersion $\le\sigma$ and
\begin{align}\label{19}
D>\text{vol}(g^n(\Omega_n(\omega)))=\int_{\Omega_n(\omega)} \left| \det \left. dg^n_z\right|_{T_z \Omega_n(\omega)} \right| d\nu(z).
\end{align}
Fix any $\theta\in K'$ and let $S_n=\left\lbrace 0\le j<n: g^j(\theta)\in K'\right\rbrace$. If $n\ge N_0$, it follows from \eqref{pesin1}, \eqref{determinante} and \eqref{17} that for $\omega\in E^{cs}(\theta)$ we have
\begin{align*}
\log \left| \det\left( \left. dg^n_z\right|_{T_z\Omega_n(\omega)} \right) \right|&=\sum_{j=0}^{n-1}\log \left| \det\left( \left. dg_{g^j(z)}\right|_{T_{g^j(z)}g^j(\Omega_n(\omega))} \right) \right| \\
&\ge \sum_{j\in S_n}\log \left| \det\left( \left. dg_{g^j(z)}\right|_{T_{g^j(z)}g^j(\Omega_n(\omega))} \right) \right|-NP(n-\text{card}S_n)\\
&\ge \sum_{j\in S_n}\log \left| \det\left( \left. dg_{g^j(\theta)}\right|_{E^u(g^j(\theta))} \right) \right|-\varepsilon n-NP(n-\text{card}S_n)\\
&\ge \sum_{j=0}^{n-1}\log \left| \det\left( \left. dg_{g^j(\theta)}\right|_{E^u(g^j(\theta))} \right) \right|-\varepsilon n-2NP(n-\text{card}S_n)\\
&=\log\left| \det \left(\left. dg^n_\theta \right|_{E^u(\theta)}  \right)\right|  -\varepsilon n-2NP(n-\text{card}S_n)\\ 
&\ge nN(\mathcal{X}^+(\theta)-\varepsilon)-\varepsilon n-2NP(n-\text{card}S_n)\\
&\ge  nN(\mathcal{X}^+(\theta)-\varepsilon)-\varepsilon n-4NPn\sqrt{\varepsilon}.
\end{align*}
From \eqref{19} we obtain that
\begin{align*}
D>\nu(\Omega_n(\omega))\cdot \exp\left( nN(\mathcal{X}^+(\theta)-\varepsilon)-\varepsilon n-4NPn\sqrt{\varepsilon}\right) 
\end{align*}
for every $\theta\in K'$ and $\omega\in E^{cs}(\theta)$. It follows from \eqref{vol} that
\begin{align*}
\nu(S_n(g,\rho,\theta))\le B\cdot D\cdot\exp\left( -nN(\mathcal{X}^+(\theta)-\varepsilon)+\varepsilon n+4NPn\sqrt{\varepsilon}\right) .
\end{align*}
Therefore, for every $\theta\in K'$,
\begin{align*}
h_\nu(g,\rho,\theta)=\limsup_{n\rightarrow\infty}-\dfrac{1}{n}\log \nu(S_n(g,\rho,\theta))\ge N\left( \mathcal{X}^+(\theta)-\varepsilon-\dfrac{\varepsilon}{N}-4P\sqrt{\varepsilon}\right). 
\end{align*}
This completes the proof of the proposition.
\end{proof}
We will show that the function $\rho$ of Proposition 6.2 allows us to find a lower bound for the entropy of $\phi^N$. To prove this, Mañé constructed a partition of the manifold with certain properties using strongly the compactness condition (see Lemma 2 of \cite{mane}). Since the manifold $SM$ is not necessarily compact in our case, we will use another technique to construct a partition that satisfies the same properties. Consider the constant $a\in (0,1)$, $a<t_0/2$, used in property \eqref{determinante}.
\begin{lemma}
Let $M$ be a complete Riemannian manifold and suppose that the curvature tensor and the derivative of the curvature tensor are both uniformly bounded. For every $\theta\in SM$ we have that
\begin{align*}
\emph{diam}\,exp_\theta U\le \dfrac{5}{2}\cdot \emph{diam } U,    
\end{align*}
where $U\subset B(0,a)\subset T_\theta SM$. 
\end{lemma}
\begin{proof}
Fix $\theta\in SM$ and consider $U\subset B(0,a)\subset T_\theta SM$. We need to prove that 
\begin{align*}
 d(exp_\theta u, exp_\theta v)\le \dfrac{5}{2} \left\| u-v\right\|    
\end{align*}
for every $u,v\in U$. Consider the segment $q(t)=tu + (1-t)v$ and the curve $\gamma(t)=exp_\theta q(t)$ that joins $exp_\theta u$ with $exp_\theta v$. Then
\begin{align}\label{roro}
 l(\gamma)&=\int_{0}^{1} \left\| \gamma'(t)\right\| \nonumber \\
 &=\int_{0}^{1}\left\| d(exp_\theta)_{q(t)} (u-v)\right\| dt.    
\end{align}
For each $t\in [0,1]$, there are $w(t)\in T_\theta SM$ with $\left\| w(t)\right\|=1$ and $s(t)\in \mathbb{R}$ with $\left|s(t)\right|\le t_0$ such that
\begin{align*}
 q(t)=s(t)w(t).   
\end{align*}
Since $a\le t_0/2$, from Proposition 2.2 we have that
\begin{align*}
 \left\| d(exp_\theta)_{q(t)} (u-v)\right\|&=\left\| d(exp_\theta)_{s(t)w(t)} (u-v)\right\|   \\
&\le \dfrac{5}{2} \left\|u-v\right\|.
\end{align*}
Therefore in \eqref{roro}
\begin{align*}
d(exp_\theta u, exp_\theta v)\le l(\gamma)\le  \dfrac{5}{2} \left\|u-v\right\|   
\end{align*}
completing the proof. 
\end{proof}
Consider the function $\rho:SM\rightarrow (0,1)$ defined in \eqref{rho}.
\begin{lemma}
There exists a countable partition $\mathcal{P}$ of $SM$ with finite entropy such that, if $\mathcal{P}(\theta)$ denotes the atom of $\mathcal{P}$ containing $\theta$, then
	\begin{align*}
		\emph{diam } \mathcal{P}(\theta)\le \rho(\theta)
	\end{align*}
	for $\mu$-almost every $\theta\in SM$.
\end{lemma}
\begin{proof} 
For each $n\ge 0$, define
\begin{align*}
	U_n=\left\lbrace \theta\in SM: e^{-(n+1)}<\rho(\theta)\le e^{-n} \right\rbrace .
\end{align*}
Since $\log \rho\in L^1(SM,\mu)$, we have that
\begin{align*}
	\sum_{n=0}^\infty n \mu(U_n)\le -\sum_{n=0}^\infty \quad\int_{U_n}\log \rho(\theta)d\mu(\theta)= -\int_{SM}\log \rho(\theta)d\mu(\theta)<\infty.
\end{align*}
Then, by Lemma 1 of \cite{mane} we obtain
\begin{align}\label{un}
	\sum_{n=0}^\infty \mu(U_n)\log \mu(U_n)<\infty.
\end{align}
For $\theta\in SM\setminus K'$ we have that $\rho(\theta)=a$. Then there exists $n_0\ge 0$ such that $$e^{-(n_0+1)}<a\le e^{-n_0}$$ and $U_n\cap (SM\setminus K')=\emptyset$ for every $n\neq n_0$. This implies that $U_n\subset K'$ for every $n\neq n_0$. Define
\begin{align*}
	U_{n_0}^*=U_{n_0}\cap K'.
\end{align*} 
Since $K'$ is compact, there exist $A>0$ and $r_0>0$ such that for all $0<r\le r_0$, there exists a partition $\mathcal{Q}_r$ of $K'$ whose atoms have diameter less than or equal to $r$ and such that the number of atoms in $\mathcal{Q}_r$, denoted by $\left| \mathcal{Q}_r\right|$, satisfies
\begin{align*}
\left| \mathcal{Q}_r\right|\le A\left( \dfrac{1}{r}\right)^{\dim (SM)}. 
\end{align*} 
Define $\mathcal{Q}$ as the partition of $K'$ given by
\begin{itemize}
    \item[.] Sets $X\cap U_n$, for $n\ge 0$, $n\neq n_0$, where $X\in \mathcal{Q}_{r_n}$ and $r_n=e^{-(n+1)}$ such that $\mu(X\cap U_n)>0$.
    \item[.] Sets $X\cap U_{n_0}^*$, where $X\in \mathcal{Q}_{r_{n_0}}$ and $r_{n_0}=e^{-(n_0+1)}$ such that $\mu(X\cap U_{n_0}^*)>0$.
\end{itemize}
On the other hand, consider $0<\varepsilon'<a/10$ such that, we can choose a measurable set (like a ``ring" covering $SM\setminus K'$) 
\begin{align*}
	V_1\subseteq\left\lbrace \theta\in SM\setminus K': d(\theta,K')\le \varepsilon'\right\rbrace:=E_1 
\end{align*}
that satisfies $$\mu(V_1)\le \sqrt{\varepsilon}.$$
Define $K'_1=K'\cup V_1$ and choose a measurable set (like a ``ring" covering $SM\setminus K'_1$)
\begin{align*}
	V_2\subseteq\left\lbrace \theta\in SM\setminus K'_1: d(\theta,K'_1)\le \varepsilon'\right\rbrace:=E_2 
\end{align*}
that satisfies $$\mu(V_2)\le \dfrac{\sqrt{\varepsilon}}{2}.$$ 
Proceeding inductively, we define bounded measurable sets 
\begin{align*}
	V_n\subseteq \left\lbrace \theta\in SM\setminus K'_{n-1}: d(\theta,K'_{n-1})\le \varepsilon'\right\rbrace:=E_n, 
\end{align*}
where $K'_{n-1}=K'\cup V_1\ldots\cup V_{n-1}$, with measure $$\mu(V_n)\le \dfrac{\sqrt{\varepsilon}}{2^{n-1}}.$$
Since
\begin{align*}
\sum_{n=1}^\infty n\mu(V_n)\le \sum_{n=1}^\infty\dfrac{n}{2^{n-1}}\cdot  \sqrt{\varepsilon}<\infty,
\end{align*}
by Lemma 1 of \cite{mane} we have that
\begin{equation}\label{vn}
\sum_{n=1}^\infty \mu(V_n)\log \mu(V_n)<\infty.
\end{equation}
Let $k$ be the number of balls of radius $a/10$ which cover $E_1$ and denote by $ B(\theta_1,a/10),\ldots,$\linebreak$ B(\theta_k,a/10)$ this covering. We claim that
\begin{align*}
E_2\subseteq \bigcup_{i=1}^k B(\theta_i,a/5).
\end{align*}
In fact, suppose that exists $\theta\in E_2$ such that $d(\theta,\theta_i)\ge a/5$, for every $i=1,\ldots,k$. By construction, there is $\omega\in E_1$ such that 
\begin{align}\label{contrad}
 d(\theta,\omega)\le \varepsilon'<\dfrac{a}{10}.  
\end{align}
Since we cover $E_1$ by balls, $\omega\in B(\theta_{i_0},a/10)$ for some $i_0\in \{ 1,\ldots,k\}$. Therefore,
\begin{align*}
d(\theta,\omega)&\ge d(\theta,\theta_{i_0})-d(\theta_{i_0},\omega)\\
&> \dfrac{a}{5} -\dfrac{a}{10}\\
&=\dfrac{a}{10},
\end{align*}
which is a contradiction with \eqref{contrad}. This proves the claim. Since $SM$ is complete (see Lemma 2.1), for each $i\in \{1,\ldots,k\}$, there is an open ball $B^i(0,a/5)\subset T_{\theta_i}SM$ such that $$exp_{\theta_i}(B^i(0,a/5))=B(\theta_i,a/5).$$ By \cite{verger} there exists $N_1:=N_1(a)>0$, which depends on the dimension of $SM$ and $a$, such that the minimal number of balls of radius $a/10$ which can cover $B^i(0,a/5)$ is bounded by $N_1$. Suppose that
\begin{align*}
B^{i}_{1},\ldots, B^{i}_{N_1}
\end{align*}
are balls of radius $a/10$ that cover $B^i(0,a/5)$. From Lemma 6.3, if we project these balls to the manifold $SM$ by the exponential map we have that $exp_{\theta_i}B^{i}_{j}$ are sets of diameter
\begin{align*}
 \text{diam } exp_{\theta_i} B^{i}_{j}\le \dfrac{5}{2}\text{diam }B^{i}_{j}=\dfrac{5}{2}\cdot \dfrac{2a}{10}=\dfrac{a}{2}.  
\end{align*}
Then we can cover $E_2$ by $k N_1$ sets of diameter $\le a/2$. Since every set of diameter $\le a/2$ is contained in a ball of radius $a/2$, we can cover $E_2$ by $kN_1$ balls $B(\omega_1,a/2),\ldots, B(\omega_{kN_1},a/2)$. Analogously, since $\varepsilon'<a/10$ we have that
\begin{align*}
E_3\subset \bigcup_{i=1}^{kN_1} B(\omega_i,6a/10).    
\end{align*}
For each $i\in \{1,\ldots,kN_1\}$, there is an open ball $B^i(0,6a/10)\subset T_{\omega_i}SM$ such that $$exp_{\omega_i}(B^i(0,6a/10))=B(\omega_i,6a/10).$$ By \cite{verger} there exists $N_2:=N_2(a)>0$, which depends on the dimension of $SM$ and $a$, such that the minimal number of balls of radius $a/10$ which can cover $B^i(0,6a/10)$ is bounded by $N_2$ and repeating the previous process we have that we can cover $E_3$ by $kN_1N_2$ balls of radius $a/2$. Continuing inductively, we obtain that $E_n$ can be covered by $kN_1N_2^{n-2}$ balls of radius $a/2$. Therefore, for every $n\ge 1$, define a partition $\hat{\mathcal{P}}_n$ of $V_n$ whose atoms have diameter $\le a$ and the number of atoms satisfies
\begin{align*}
| \hat{\mathcal{P}}_1 |	\le k, \hspace{0.6cm}|\hat{\mathcal{P}}_n |\le kN_1N_2^{n-2}, \hspace{0.3cm}\forall n\ge 2. 
\end{align*}
Finally, define the partition of $SM$ as $$\mathcal{P}=\mathcal{Q}\cup \bigcup_{n\ge 1} \hat{\mathcal{P}}_n.$$
Recalling the well-known inequality
\begin{align*}
-\sum_{i=1}^m x_i\log x_i\le \left(\sum_{i=1}^m x_i\right)\left( \log m -\log \sum_{i=1}^m x_i\right)  
\end{align*}
which holds for any set of real numbers $0<x_i\le 1$, $i=1,\ldots,m$. We claim that $H(\mathcal{P})<+\infty$. In fact, from \eqref{un} and \eqref{vn} we obtain that
\begin{align*}
H(\mathcal{P})&=\sum_{n\ge 0,\ n\neq n_0} \left( -\sum_{P\in \mathcal{Q},P\subset U_n} \mu(P)\log \mu(P)\right)+\left( -\sum_{P\in \mathcal{Q},P\subset U_{n_0}^*} \mu(P)\log \mu(P)\right)\\
&\hspace{0.5cm}+ \sum_{n\ge 1} \left( -\sum_{P\in \hat{\mathcal{P}}_n} \mu(P)\log \mu(P)\right)\\
&\le \sum_{n\ge 0,\ n\neq n_0} \mu(U_n)\left[ \log \left| \mathcal{Q}_{r_n} \right| - \log\mu(U_n) \right] +\mu(U_{n_0}^*)\left[ \log | \mathcal{Q}_{r_{n_0}} | - \log\mu(U_{n_0}^*) \right]\\ 
&\hspace{0.5cm}+\sum_{n\ge 1,} \mu(V_n)\left[ \log | \hat{\mathcal{P}}_n | - \log\mu(V_n) \right]
\end{align*}

\begin{align*}
\hspace{0.8cm}&\le \sum_{n\ge 0,\ n\neq n_0} \mu(U_n)\left[\log A +  \dim (SM)(n+1) - \log\mu(U_n)\right]  \\
& \hspace{0.6cm} +\mu(U_{n_0}^*)\left[\log A+  \dim (SM)(n_0+1)- \log\mu(U_{n_0}^*)\right]+\mu(V_1)\left[\log k -\log \mu(V_1)\right]  \\
&\hspace{0.6cm}+ \sum_{n\ge 2} \mu(V_n)\left[\log k + \log N_1 + (n-2)\log N_2- \log\mu(V_n)\right]\\
&<\infty.
\end{align*}
Moreover, if $\theta\in U_n$, for $n\ge 0, n\neq n_0$, then $\mathcal{P}(\theta)$ is contained in an atom of $\mathcal{Q}_{r_n}$ and
\begin{align*}
	\text{diam } \mathcal{P}(\theta)\le r_n= e^{-(n+1)}<\rho(\theta).
\end{align*}
 If $\theta\in U_{n_0}^*$, then $\mathcal{P}(\theta)$ is contained in an atom of $\mathcal{Q}_{r_{n_0}}$ and
\begin{align*}
	\text{diam } \mathcal{P}(\theta)\le r_{n_0}= e^{-(n_0+1)}<\rho(\theta).
\end{align*}
In another case, if $\theta\in V_n$, for $n\ge 1$, then $\mathcal{P}(\theta)$ is contained in an atom of $\mathcal{\hat{P}}_n$ and $\text{diam} \mathcal{P}(\theta)\le a=\rho(\theta)$. 
\end{proof}
Given that $M$ has finite volume, it follows that $SM$ also has finite volume. Lemma 6.4, together with the Radon-Nikodym Theorem and Shannon-McMillan-Breiman Theorem, allow us to obtain the following result proved in \cite{mane}.
\begin{proposition}  If $\mu\ll\nu$, where $\nu$ denotes the Lebesgue measure on $SM$, then 
\begin{align*}
	h_\mu(\phi^N)\ge \int\limits_{SM} h_\nu(\phi^N,\rho,\theta)d\mu(\theta).
\end{align*}  
\end{proposition}
\textbf{Proof of Theorem 1.2.} We just need to prove that
\begin{align*}
 h_\mu(\phi) \ge  \int_{SM} \mathcal{X}^+(\theta)d\mu(\theta). 
\end{align*}
Consider $\displaystyle \Upsilon=\sup_{\theta\in SM}\left\lbrace \left\|d\phi^1_\theta \right\|, \left\|d\phi^{-1}_\theta \right\|  \right\rbrace$. Then
\begin{align*}
\int\limits_{SM\setminus {K'}}\mathcal{X}^+(\theta)d\mu(\theta)&\le \mu(SM\setminus K')\cdot\dim (SM)\cdot\log \Upsilon\\
	&\le 2\sqrt{\varepsilon} \cdot\dim (SM)\cdot\log \Upsilon.
\end{align*}
From Propositions 6.2 and 6.5 we have that
\begin{align*}
	h_\mu(\phi^N)&\ge \int_{SM}h_\nu(\phi^N,\rho,\theta)d\mu(\theta)\\
	&\ge \int_{K'} h_\nu(\phi^N,\rho,\theta)d\mu(\theta)\\
	&\ge N\int_{K'} \mathcal{X}^+(\theta)d\mu(\theta)-N\varepsilon-\varepsilon-4NP\sqrt{\varepsilon}\\
	&\ge N\int_{SM} \mathcal{X}^+(\theta)d\mu(\theta)-2\sqrt{\varepsilon}N\cdot\dim (SM)\cdot\log \Upsilon-N\varepsilon-\varepsilon-4NP\sqrt{\varepsilon}.
\end{align*}
Hence,
\begin{align*}
	h_\mu(\phi)\ge \int_{SM} \mathcal{X}^+(\theta)d\mu(\theta)-2\sqrt{\varepsilon}\cdot\dim (SM)\cdot\log \Upsilon- \varepsilon-\dfrac{\varepsilon}{N}-4P\sqrt{\varepsilon}.
\end{align*}
Letting $\varepsilon\rightarrow 0$ we obtain the desired lower bound.$\hfill\square$
\subsection*{Acknowledgments} Alexander Cantoral thanks FAPERJ for partially supporting the research (Grant  E-26/202.303/2022). Sergio Romaña thanks ``Bolsa Jovem Cientista do Nosso Estado No. E-26/201.432/2022",  NNSFC 12071202, and NNSFC 12161141002 from China. The second author thanks the Department of Mathematics of the SUSTech- China for its hospitality during the execution of this work.
	\bibliographystyle{abbrv}
	
	\bibliography{references} 

\begin{thebibliography}{10}

\bibitem{pesin}
L.~Barreira and Y.~B. Pesin.
\newblock {\em Nonuniform hyperbolicity: Dynamics of systems with nonzero
  Lyapunov exponents}, volume 115.
\newblock Cambridge University Press Cambridge, 2007.

\bibitem{bolton}
J.~Bolton.
\newblock Conditions under which a geodesic flow is {A}nosov.
\newblock {\em Math. Ann.}, 240(2):103--113, 1979.

\bibitem{burns}
K.~Burns, H.~Masur, and A.~Wilkinson.
\newblock The {W}eil-{P}etersson geodesic flow is ergodic.
\newblock {\em Ann. of Math. (2)}, 175(2):835--908, 2012.

\bibitem{eber}
P.~Eberlein.
\newblock When is a geodesic flow of anosov type? i.
\newblock {\em Journal of Differential Geometry}, 8(3):437--463, 1973.

\bibitem{ipsen}
I.~C. Ipsen and C.~D. Meyer.
\newblock The angle between complementary subspaces.
\newblock {\em The American mathematical monthly}, 102(10):904--911, 1995.

\bibitem{klin}
W.~Klingenberg.
\newblock Riemannian manifolds with geodesic flow of {A}nosov type.
\newblock {\em Ann. of Math. (2)}, 99:1--13, 1974.

\bibitem{kowalski}
O.~Kowalski and M.~Sekizawa.
\newblock On tangent sphere bundles with small or large constant radius.
\newblock volume~18, pages 207--219. 2000.
\newblock Special issue in memory of Alfred Gray (1939--1998).

\bibitem{lee}
J.~M. Lee.
\newblock {\em Introduction to Riemannian manifolds}, volume~2.
\newblock Springer, 2018.

\bibitem{liao}
G.~Liao and N.~Qiu.
\newblock Margulis--ruelle inequality for general manifolds.
\newblock {\em Ergodic Theory and Dynamical Systems}, 42(6):2064--2079, 2022.

\bibitem{mane2}
R.~Ma\~{n}\'{e}.
\newblock On a theorem of {K}lingenberg.
\newblock {\em Dynamical systems and bifurcation theory ({R}io de {J}aneiro,
  1985)}, 160:319--345, 1987.

\bibitem{mane}
R.~Ma{\~n}{\'e}.
\newblock A proof of {P}esin's formula.
\newblock {\em Ergodic Theory and Dynamical Systems}, 1(1):95--102, 1981.

\bibitem{nocon}
{\'I}.~Melo and S.~Roma{\~n}a.
\newblock {R}iemannian manifolds with {A}nosov geodesic flow do not have
  conjugate points.
\newblock {\em arXiv preprint arXiv:2008.12898}, 2020.

\bibitem{rigidity}
{\'I}.~Melo and S.~Roma{\~n}a.
\newblock Some rigidity theorems for {A}nosov geodesic flows in manifolds of
  finite volume.
\newblock {\em Qualitative Theory of Dynamical Systems}, 23(3):114, 2024.

\bibitem{oseledec}
V.~I. Oseledec.
\newblock A multiplicative ergodic theorem. {C}haracteristic {L}japunov,
  exponents of dynamical systems.
\newblock {\em Trudy Moskov. Mat. Ob\v{s}\v{c}.}, 19:179--210, 1968.

\bibitem{paternain}
G.~P. Paternain.
\newblock {\em Geodesic flows}, volume 180 of {\em Progress in Mathematics}.
\newblock Birkh\"{a}user Boston, Inc., Boston, MA, 1999.

\bibitem{igual}
Y.~B. Pesin.
\newblock Characteristic {L}yapunov exponents and smooth ergodic theory.
\newblock {\em Russian Mathematical Surveys}, 32(4):55, 1977.

\bibitem{contra}
F.~Riquelme.
\newblock Counterexamples to ruelle’s inequality in the noncompact case.
\newblock In {\em Annales de l'institut Fourier}, volume~67, pages 23--41,
  2017.

\bibitem{riquelme}
F.~Riquelme.
\newblock Ruelle's inequality in negative curvature.
\newblock {\em Discrete Contin. Dyn. Syst.}, 38(6):2809--2825, 2018.

\bibitem{ruelle}
D.~Ruelle.
\newblock An inequality for the entropy of differentiable maps.
\newblock {\em Boletim da Sociedade Brasileira de
  Matem{\'a}tica-Bulletin/Brazilian Mathematical Society}, 9(1):83--87, 1978.

\bibitem{verger}
J.-L. Verger-Gaugry.
\newblock Covering a ball with smaller equal balls in $\mathbb{R}^n$.
\newblock {\em Discrete \& Computational Geometry}, 33:143--155, 2005.

\end{thebibliography}
	
\end{document}